\documentclass[onefignum,onetabnum]{siamart171218}

\usepackage{braket,amsfonts}
\usepackage{multirow}

\usepackage{array}

\usepackage[caption=false]{subfig}
\usepackage{pgfplots}

\newsiamthm{claim}{Claim}
\newsiamremark{remark}{Remark}
\newsiamremark{hypothesis}{Hypothesis}
\crefname{hypothesis}{Hypothesis}{Hypotheses}
\newsiamremark{assumption}{Assumption}
\newsiamremark{example}{Example}

\usepackage{algorithmic}

\usepackage{graphicx,epstopdf}

\usepackage{amsopn}

\usepackage{xspace}
\usepackage{bold-extra}
\usepackage[most]{tcolorbox}

\usepackage{dsfont}

\newcommand{\T}{\mathrm{T}}
\newcommand{\Tr}{\mathrm{Tr}}
\newcommand{\etal}{et.al. }

\title{A Riemannian smoothing steepest descent method for non-Lipschitz optimization on submanifolds \thanks{Submitted to the editors on \today.
\funding{
C. Zhang was supported in part by Natural Science Foundation of Beijing (No. 1202021).
X. Chen was supported in part by Hong Kong Research Council Grant PolyU15300219. S. Ma was supported in part by NSF grants DMS-1953210 and CCF-2007797, and UC Davis CeDAR (Center for Data Science and Artificial Intelligence Research) Innovative Data Science Seed Funding Program.
}}}

\author{Chao Zhang\thanks{Department of Applied Mathematics,
Beijing Jiaotong University, Beijing 100044, China. (\email{zc.njtu@163.com}).}
\and Xiaojun Chen\thanks{Department of Applied Mathematics,
The Hong Kong Polytechnic University, Hung Hom, Kowloon, Hong Kong, China. (\email{xiaojun.chen@polyu.edu.hk}).}
\and  Shiqian Ma\thanks{Department of Mathematics, University of California, Davis, CA 95616, USA. (Corresponding author. \email{sqma@ucdavis.edu}).}
}

\headers{Riemannian smoothing steepest descent method}{Chao Zhang, Xiaojun Chen, and Shiqian Ma}

\ifpdf
\hypersetup{ pdftitle={Guide to Using  SIAM'S \LaTeX\ Style} }
\fi

\begin{document}
\maketitle

\begin{abstract}
In this paper, we propose a Riemannian smoothing steepest descent method to minimize a nonconvex and non-Lipschitz function on submanifolds. The generalized subdifferentials on Riemannian manifold and the Riemannian gradient sub-consistency are defined and discussed. We prove that any accumulation point of the sequence generated by the Riemannian smoothing steepest descent method is a stationary point associated with the smoothing function employed in the method, which is necessary for the local optimality of the original non-Lipschitz problem. Under the Riemannian gradient sub-consistency condition, we also prove that any accumulation point is a Riemannian limiting stationary point of the original non-Lipschitz problem. Numerical experiments are conducted to demonstrate the efficiency of the proposed method.
\end{abstract}
 
\begin{keywords}
Riemannian submanifold, Non-Lipschitz, Smoothing steepest descent method, Riemannian generalized subdifferentials, Riemannian gradient sub-consistency
\end{keywords}

\begin{AMS}
  65K10, 90C26, 90C46
\end{AMS}

\section{Introduction}
\label{sec:intro}

We consider the Riemannian optimization problem
\begin{eqnarray}\label{model}
\min \ f(x),  \quad x\in {\cal M},
\end{eqnarray}
where $\cal M$ is a complete Riemannian submanifold of ${\mathbb{R}}^n$ and $f: \mathbb{R}^n \to\mathbb{R}$ is a proper lower semi-continuous function which may be nonsmooth and non-Lipschitzian. Such problems arise in a variety of applications in signal processing, computer vision, and data mining \cite{Adler,Bacak,JuSun}.

Many classical algorithms for unconstrained and smooth optimization have been extended from the Euclidean space to the Riemannian manifolds, such as the gradient descent algorithm, the conjugate gradient algorithm, the quasi-Newton algorithm and the trust region method \cite{Absil1,Absilbook,WHuang1}. Recently, Riemannian optimization with a nonsmooth but Lipschitz continuous objective function has been considered in the literature. Here the smoothness and Lipschitz continuity are interpreted when the function in question is  considered in the ambitent Euclidean space. The Clarke subdifferential of functions over manifolds has been defined and its properties have been discussed in \cite{Hosseini1}. Several algorithms have been proposed based on the notion of Clarke subdifferential. For example, Hosseini and Uschmajew \cite{Hosseini} proposed the Riemannian gradient sampling algorithm. This algorithm approximats the subdifferential using the convex hull of transported gradients from tangent spaces of randomly generated nearby points to the tangent space of the current space. The $\epsilon$-subgradient algorithm \cite{Grohs} is a steepest descent method where the descent directions are obtained by a computable approximation of the $\epsilon$-subdifferential. The line search algorithms \cite{Hosseini2} include the nonsmooth Riemannian BFGS algorithm as a special case. For both the $\epsilon$-subgradient algorithm and the line search algorithms, either the algorithms terminate
after a finite number of iterations with the $\epsilon$-subgradient-oriented descent direction being $0$, or any accumulation point is a Clarke stationary point. Other methods for nonsmooth optimization over Riemannian manifolds include the Riemannian subgradient method \cite{Li}, the Riemannian ADMM \cite{Lai2014,Kovnatsky2016}, the manifold proximal gradient method \cite{chen2019proximal,Huang-Wei-RPG,Shiqian,Wang-stochastic-ManPG}, manifold proximal point method \cite{Chen-ManPPA-2020}, manifold proximal linear method \cite{wang-manpl-2020}, and manifold augmented Lagrangian method \cite{Chen2016,Zhu2017,Ding-2021}.

To the best of our knowledge, there do not exist optimization algorithms for solving Riemannian optimization problems with general non-Lipschitz objective functions, although the Riemannian generalized subdifferentials have been studied for nonsmooth and non-Lipschitz optimization \cite{Ledyaev}. Non-Lipschitz optimization in Euclidean space finds many important applications, including but not limited to, finding sparse solutions in singal processing and data mining \cite{ChenGuo,XuF2010,Liu-Dai-Ma-2015,Liu-Ma-Dai-Zhang-2016,Luo2018}, and neat edge in image restoration \cite{Bian2015,CNZ2012,Zeng2019}. {Smoothing methods with a proper updating scheme for the smoothing parameter are efficient for solving large-scale nonsmooth optimization in Euclidean space} \cite{Chen,CNZ2012,Niu,ChenZhou,Zhang,Zhang2}.
With a fixed smoothing parameter, one solves the smoothed problem to update the iterate. Certain strategy is then applied to decide weather and how the smoothing parameter needs to be changed.  Under the so-called gradient consistency property, it can be shown that any accumulation point of the smoothing method is a limiting stationary point of the original nonsmooth optimization problem. The gradient consistency naturally holds for smoothing functions arising in various real applications with nonsmooth and Lipschitz objective functions. Smoothing methods have been widely used to solve unconstrained non-Lipschitz optimization problems \cite{Niu,ChenZhou}, and constrained non-Lipschitz optimization with feasible region being convex sets \cite{Zhang,Zhang2}. However, minimizing a non-Lipschitz function on a nonconvex set has not been widely considered in the literature. In \cite{ChenGuo}, an augmented Lagrangian method for non-Lipschitz nonconvex programming was proposed where the constraint set is nonconvex. 

In this paper, we extend the smoothing steepest descent method in Euclidean space 
to Riemannian submanifolds. The smoothing steepest descent method is a special case of the smoothing projected gradient method for unconstrained nonsmooth optimization \cite{Zhang}. Our Riemannian smoothing steepest descent method (RSSD) uses the Riemannian gradient of the smoothing function in each iteration. Therefore, we do not need to sample points around the current point to get (sub)gradient information of the current point. This avoids the vector transport comparing with existing gradient-type algorithms such as the Riemannian gradient sampling algorithm \cite{Hosseini} and the Riemannian $\epsilon$-subgradient algorithm \cite{Grohs}. Our RSSD is easy to implement and can be shown to converge to a stationary point of the Riemannian optimization with non-Lipschitz objective.

The rest of this paper is organized as follows. In Section \ref{sec:preliminaries}, we give a brief review on some basic concepts and properties relating to Riemannian manifold, and the generalized subdifferentials for non-Lipschitz functions in Euclidean space. In Section \ref{sec:Riemannian gradient}, we define the generalized subdifferentials for non-Lipschitz functions on Riemannian submanifolds and discuss their properties. We also define and discuss the Riemannian gradient sub-consistency that is essential to show that any accumulation point of our proposed RSSD method is a Riemannian limiting stationary point. In Section \ref{sec:Riemannian RSSD}, we propose our RSSD method and analyze its convergence behavior. In Section \ref{sec:numerical}, we conduct numerical experiments on two imporant applications: finding a sparse vector in a subspace, and the sparsely-used orthogonal complete dictionary learning. Finally, we draw some concluding remarks in Section \ref{sec:conclude}.

\section{Preliminaries}\label{sec:preliminaries} 
We define some notation first. Throughout this paper, ${\cal M}$ denotes a finite dimensional complete Riemannian submanifold embeded in an Euclidean space. We consider the Riemannian metric on $\cal M$ that is induced from the Euclidean inner product; i.e., for any $\xi, \eta \in \T_x\cal M$, we have $\langle\xi,\eta\rangle_x = \Tr(\xi^\top\eta)$, where $\T_x\cal M$ denotes the tangent space of $\cal M$ at $x$, and $\Tr(Z)$ denotes the trace of matrix $Z$. The cotangent space at $x$ via the Riemannian metric is denoted as $\T_x{\cal M}^*$. We use $\T {\cal M}$ to denote the tangent bundle, i.e., the set of all tangent vectors: $\T{\cal M}:= \bigcup_{x\in {\cal M}} \T_x {\cal M}$. We use $\|x\|$ to denote the Euclidean norm when $x$ is a vector, and the Frobenius norm when $x$ is a matrix. We use $B_{x, \delta} = \{y\mid \|y-x\|\le \delta\}$ to represent a neighborhood of $x$ with radius $\delta$. For subset $D\subseteq \mathbb{R}^n$, a function $h\in C^1(D)$ means that $h$ is smooth on $D$.

An important concept in Riemannian optimization is the retraction operation and it is defined below.  

\begin{definition}\label{retraction}(Retraction, see \cite{Absilbook}). 
A retraction on a manifold $\cal M$ is a smooth mapping $R: \T{\cal M} \to {\cal M}$ with the following properties. Here $R_x$ denotes the restriction of $R$ to the tangent space $T_x {\cal M}$.
\begin{itemize}
\item[(i)] $R_x(0_x) = x$, where $0_x$ denotes the zero element of $\T_x {\cal M}$.
\item[(ii)] It holds that  
\[d R_x(0_x) = {id}_{\T_x {\cal M}},\]
where $d R_x$ is the differential of $R_x$, and ${id}_{\T_x {\cal M}}$ denotes the identity map on $\T_x {\cal M}$.  
\end{itemize}
\end{definition}

By the inverse function theorem, we know that $R_x$ is a local diffeomorphism (see, e.g., \cite{Hosseini2}). We now review some important concepts and properties related to generalized subgradients, subdifferentials {and subderivatives} of non-Lipschitz functions in Euclidean space $\mathbb{R}^n$. They are specializations of \cite[Definitions 8.3, 8.1]{Rockafellar} to our setting (note that in our case the function $f$ is finite-valued.)

\begin{definition}\label{subdif-Rn}(Subgradients). We consider a proper lower semi-continuous function $f: \mathbb{R}^n \to  \mathbb{R}$. For a vector $v\in \mathbb{R}^n$, we say that
\begin{itemize}
\item[(i)] $v$ is a regular subgradient of $f$ at $\bar x$, written as $v\in \hat \partial f(\bar x)$, if
\[
f(x) \ge f(\bar x) + \langle v, x - \bar x \rangle + o(\|x- \bar x\|),
\]
or equivalently
\[
\liminf_{x\to \bar x, \ x \ne \bar x} \frac{f(x)-f(\bar x) - \langle v, x - \bar x \rangle}{\|x - \bar x\|} \ge 0;
\]
\item[(ii)] $v$ is a (general) limiting subgradient of $f$ at $\bar x$, written as $v\in \partial f(\bar x)$, if there exist $(x^{\nu}, f(x^{\nu}))  \to (\bar x, f(\bar x))$ and $v^{\nu} \in \hat \partial f(x^{\nu})$ with $v^{\nu} \to v$;
\item[(iii)] $v$ is a horizontal subgradient of $f$ at $\bar x$, written as $v\in \partial^{\infty} f(\bar x)$, if the same conditions in (ii) hold, except that instead of $v^{\nu} \to v$ one has $\lambda^{\nu} v^{\nu} \to v$ for some sequence $\lambda^{\nu} \downarrow 0$.
\end{itemize}
\end{definition}

Here $\hat \partial f(\bar x)$, $\partial f(\bar x)$, and $\partial^{\infty} f(\bar x)$ are called the regular (Fr\'{e}chet), limiting, and horizontal subdifferentials of $f$ at $\bar x$, respectively. According to   \cite{Rockafellar}, $\partial^{\circ} f(\bar x)$ is called the Clarke subdifferential if
\[
\partial^{\circ} f(\bar x) = {\rm conv} \{\partial f(\bar x) + \partial^{\infty} f(\bar x)\},
\]
where ${\rm conv}$ is the convex hull.

\begin{definition}\label{def:subderivative}(Subderivative).
For a proper lower semi-continuous function $f: \mathbb{R}^n \to  \mathbb{R}$,
the subderivative function $df(\bar x) : \mathbb{R}^n \to \mathbb{\bar{R}}$ is defined as
\[
df(\bar x)(\bar w) := \liminf_{\tau\downarrow 0,\ w\to \bar w} \frac{f(\bar x + \tau w) - f(\bar x)}{\tau}.
\]
\end{definition}

We have the two equivalent characterizations for the regular subdifferential in the following two propositions, coming from \cite[Exercise 8.4, pp. 301;  Proposition 8.5, pp. 302]{Rockafellar}. 

\begin{proposition}\label{ec-Rn}(Regular subgradient from subderivative).  
It holds that 
\[
\hat \partial f(\bar x) = \{ v \mid \langle v, w\rangle \le df(\bar x)(w)\quad \mbox{for all}\ w \}.
\]
\end{proposition}

\begin{proposition}\label{prop2.2}(Variational description of regular subgradients). A vector $v$ belongs to $\hat \partial f(\bar x) \Longleftrightarrow$ in some neighborhood of $\bar x$, there is a function $h \le f$ with $h(\bar x) = f(\bar x)$ such that $h$ is differentiable at $\bar x$ with $\nabla h(\bar x) = v$. Moreover $h$ can be smooth with $h(x)< f(x)$ for all $x \ne \bar x$ near $\bar x$.
\end{proposition}

In the case that $f: {\cal M}\to \mathbb{R}$ is a nonsmooth but locally Lipschitz continuous function, the Clarke subdifferential has also been studied and used in analyzing the convergence of algorithms, see e.g., \cite{Hosseini}. Let
\[
\Omega_f := \{x\in {\cal M}\mid f \ \mbox{is differentiable at}\ x\}.
\]
The Riemannian Clarke subdifferential, denoted by $\partial^{\circ}_{\cal R} f(x)$, is defined as \cite{Hosseini}
\begin{eqnarray}\label{Clarke-sub}
\partial_{\cal R}^{\circ} f(x):= {\rm conv}\left\{\lim_{\ell \to \infty} {\rm grad} f(x_{\ell}) \mid x_{\ell } \to x,\ x_{\ell}\in \Omega_f   \right\},
\end{eqnarray}
where grad dentoes the Riemannian gradient.
An alternative definition of $\partial_{\cal R}^{\circ} f(x)$   \cite{Hosseini} relying on the definition of
subdifferential on linear spaces is
\begin{eqnarray*}
\partial_{\cal R}^{\circ} f(x) = \partial^{\circ}(f\circ R_x)(0_x)
\end{eqnarray*}
for any retraction $R$.

A definition of generalized subdifferentials for nonsmooth non-Lipschitz function on manifold is given as follows by \cite[Definition 3.1]{Ledyaev}.

\begin{definition}\label{LeDef} 
Let $f : {\cal M} \to  \mathbb{R}$ be any  lower semicontinuous function.
The Riemannian Fr\'{e}chet  subdifferential of $f$ at $x\in {\cal M}$ is defined as
\[
\partial_{F} f(x) := \{{dh}(x)\mid  h\in C^1({\cal M}) \ \mbox{and} \ f-h \ \mbox{attains a local minimum at}\ x\},
\]
where $dh(x)$ is the differential of $h$ at $x\in {\cal M}$. 
The Riemannian limiting subdifferential of $f$ at $x\in {\cal M}$ is defined as
\[
\partial f(x):= \{\lim_{\ell \to \infty} v_{\ell}^*\mid v_{\ell}^* \in \partial_{F} f(x_{\ell}),\ (x_{\ell},f(x_{\ell})) \to (x, f(x))\}.
\]
The Riemannian horizontal subdifferential of $f$ at $x\in {\cal M}$ is defined as
\[
\partial^{\infty} f(x):= \{\lim_{\ell \to \infty} t_{\ell} v_{\ell}^*\mid v_{\ell}^* \in  \partial_{F} f(x),\ (x_{\ell},f(x_{\ell})) \to (x, f(x)) \ \mbox{and}\ t_{\ell} \downarrow 0\}.
\]
\end{definition}

Let $h$ be a $C^1(\cal M)$ function at $x$. The differential of $h$ at $x$,  $dh(x) \in \T_x{\cal M}^*$, is an element of $\T_x{\cal M}^*$, which is defined as
\[
 d h(x)(v)  = \langle {\rm grad}h(x),v\rangle, \quad \forall v\in \T_x {\cal M},
\]
where ${\rm grad } h(x)$ is the Riemannian gradient of $h$ at $x\in {\cal M}$.

We use the following definition of a smoothing function on $\mathbb{R}^n$ as in \cite{Zhang2}.

\begin{definition}\label{sf}(Smoothing function).
A function $\tilde f(\cdot,\cdot) : \mathbb{R}^n\times \mathbb{R}_+ \to \mathbb{R}$ is called a smoothing function of $f : \mathbb{R}^n \to \mathbb{R}$, if $\tilde f(\cdot,\mu)$ is continuously differentiable in ${\mathbb{R}}^n$ for any $\mu \in \mathbb{R}_{++}$,  
\begin{eqnarray}\label{sfcond1}
\lim_{z\to x,\ \mu\downarrow 0} \tilde f(z, \mu) = f(x),
\end{eqnarray}
and there exist a constant $\kappa>0$ and a function $\omega: R_{++} \to R_{++}$ such that
\begin{eqnarray}\label{sfcond2}
|\tilde f(x,\mu) - f(x)| \le \kappa \omega(\mu)\quad \mbox{with}\quad \lim_{\mu \downarrow 0} \omega(\mu) = 0.
\end{eqnarray}
\end{definition}
In order to emphasize that $\mu$ is a smoothing parameter, we sometimes also write $\tilde f(\cdot,\mu)$ as $\tilde f_{\mu}(\cdot)$ in this paper.
 
\begin{example}
We use the absolute value function $|t|, t\in \mathbb{R}$ as an example to illustrate the smoothing function. We can use the so-called uniform smoothing function
\begin{eqnarray}\label{uniform-sf}
s_{\mu}(t) = \left\{\begin{array}{ll}
                     |t| & {\rm if}\ |t|\ge \frac{\mu}{2}\\
                     \frac{t^2}{\mu} + \frac{\mu}{4} & {\rm if}\ |t|<\frac{\mu}{2},
                     \end{array}
              \right. 
\end{eqnarray} 
with $\kappa = \frac{1}{4}$ and $\omega(\mu) = \mu$ in \eqref{sfcond2}.

We refer to \cite{Chen} for more examples of smoothing functions. For non-Lipschitz term $|t|^p$ where $0<p<1$, its smoothing function can be defined as $(s_{\mu}(t))^p$, with $\kappa = (\frac{1}{4})^p$ and $\omega(\mu) = \mu^p$ in \eqref{sfcond2}.
\end{example}

\section{Riemannian generalized subdifferentials and Riemannian gradient sub-consistency}\label{sec:Riemannian gradient}

In this section, we define and discuss several genearlized subdifferentials, Riemannian gradient sub-consistency of proper lower semicontinuous functions, and related stationary points of \eqref{model}. These concepts play important roles in the  convergence analysis of our proposed method in the next section.

\subsection{Riemanian generalized subdifferentials}
Motivated by the generalized Clarke subdifferential on Riemannian manifold in \eqref{Clarke-sub}, and the generalized subdifferentials for a lower semicontiunouos function on Riemannian manifold in Definition \ref{LeDef} given by Ledyaev and Zhu \cite{Ledyaev}, we define the generalized subdifferentials for lower semicontinuous functions. {Similar as \cite{Yang} for the nonsmooth but Lipschitz case, we define the generalized subdifferentials on the tangent space, not on the cotangent space as in Definition \ref{LeDef} by \cite{Ledyaev}. Since the Riemannian gradient of a differentiable function is defined on the tangent space, from the computational point of view, we find that it is more reasonable to define the generalized subdifferential of a nonsmooth function on the tangent space. }

\begin{definition}\label{mydef}
Let $f : \mathbb{R}^n \to  \mathbb{R}$ be a lower semicontinuous function. The Riemannian Fr\'{e}chet subdifferential of $f$ at $x\in {\cal M}$ is defined as
\begin{eqnarray}\label{hat_partial_R}
\quad\quad\
 \hat \partial_{\cal R} f(x) :=\{\operatorname{grad} h(x)\mid \ & &  \exists\ \delta>0 \ \mbox{such that}\ h\in C^1(B_{x, \delta}) \ \mbox{and} \  \quad\quad\quad\quad\\
 & &\quad
  f-h\ \mbox{attains a local minimum at}
\ x\ \mbox{on}\ {\cal M} \}.  \nonumber
 \end{eqnarray}
  The Riemannian limiting subdifferential of $f$ at $x\in {\cal M}$ is defined as
 \begin{eqnarray}
 \partial_{\cal R} f(x):= \{\lim_{\ell \to \infty} v_{\ell}^*\mid v_{\ell}^* \in \hat \partial_{\cal R} f(x_{\ell}), (x_{\ell},f(x_{\ell})) \to (x, f(x))\}.
 \end{eqnarray}
 The Riemannian horizontal subdifferential of $f$ at $x\in {\cal M}$ is defined as
\begin{eqnarray}
 \partial^{\infty}_{\cal R} f(x):= \{\lim_{\ell \to \infty} t_{\ell} v_{\ell}^*\mid v_{\ell}^* \in \hat \partial_{\cal R} f(x_{\ell}), (x_{\ell},f(x_{\ell})) \to (x, f(x)) \ \mbox{and}\ t_{\ell} \downarrow 0\}.
 \end{eqnarray}
\end{definition}
 
The Riemannian regular subdifferential is essentially only related to the local property of $h$. 
By Whitney extension theorem \cite{Whitney}, any smooth function on 
$B_{x, \delta} \cap {\cal M}$ can be extended on the whole Euclidean space ${\mathbb{R}}^n$. When ${\cal M} = \mathbb{R}^n$, the Riemannian Fr\'{e}chet, limiting, and horizontal subdifferentials coincide with the usual Fr\'{e}chet, limiting, and horizontal subdifferentials in $\mathbb{R}^n$. When $f$ is Lipschitz continuous, we know that $\partial^{\infty}_{\cal R} f(x)= \{0\}$, and then the Riemannian Clarke subdifferential is
\[
\partial_{\cal R}^{\circ} f(x) = {\rm conv} \{\partial_{\cal R} f(x)\},
\]
which is widely used in the Riemannian optimization literature \cite{Grohs,Hosseini1,Hosseini,Hosseini2}.

We make a brief comparison and build up the relation between Definition \ref{mydef}
and Definition \ref{LeDef}. First we note that $\hat \partial_{\cal R}  f(x) \subseteq \T_x {\cal M} $, and $\partial_F f(x) \subseteq \T_x {\cal M}^*$. There is a one-to-one correspondence between element ${\rm grad}h(x) \in \hat \partial_{\cal R} f(x)$ and $dh(x) \in \partial_F f(x)$. That is, for any $dh(x)\in \partial_F f(x)$, there is a unique $\operatorname{grad} h(x)\in \hat \partial_{\cal R} f(x)$ that corresponds to it.
Moreover, we have
\begin{eqnarray}\label{relation}
dh(x)(\cdot) := \langle \operatorname{grad} h(x), \cdot \rangle,
\end{eqnarray}
because for any $x\in {\cal M}$ and $\xi \in \T_x {\cal M}$, 
\[
\langle \operatorname{grad} h(x), \xi \rangle =  dh(x)(\xi) = \left.\frac{dh(\gamma(t))}{dt} \right|_{t=0},
\]
where $\gamma$  is a curve on $\cal M$ with $\gamma(0) = x$ and $\dot{\gamma}(0) = \xi$.

Using Definition \ref{hat_partial_R}, and the facts that $\cal M$ is a submanifold embedded in $\mathbb{R}^n$ and $h\in C^1(B_{x,\delta})$, we have
\[
\operatorname{grad} h(x) = \operatorname{Proj}_{\T_x {\cal M}} \nabla h(x),
\]
where $\operatorname{Proj}_{\T_{x} {\cal M}} y$ denotes the projection of $y\in \mathbb{R}^n$ onto $\T_x {\cal M}$. Consequently,
\begin{eqnarray}\label{hat_partial_R_1}
\quad\quad\partial_{\cal R} f(x)  = \{\operatorname{Proj}_{\T_{ x} {\cal M}}\nabla h(x) \mid  & &\  \exists\ \delta>0\ \mbox{such that}\ h\in C^1(B_{x,\delta})\ \mbox{and}  \\
            & &\quad\quad \ f-h
 \ \mbox{attains a local minimum at}\ x\ \mbox{on}\ {\cal M}\}. \nonumber
\end{eqnarray}
Note that for any $v \in \hat \partial f(x)$, according to Proposition \ref{prop2.2}, there exists $h\in C^1$, such that $f-h$ attains a local minimum at $x$ on $\mathbb{R}^n$, which is sure to attain a local minimum at $x$ on ${\cal M} \subseteq \mathbb{R}^n$. This, combining with \eqref{hat_partial_R_1}, indicates that
\begin{eqnarray}\label{hat_partial_R_2}
\hat \partial_{\cal R} f(x) \supseteq \{\operatorname{Proj}_{\T_{ x} {\cal M}}v\mid v\in \hat \partial f(x)\}.
\end{eqnarray}
By using Definition \ref{mydef}, we have $\hat \partial_{\cal R} f(x) \subseteq
\partial_{\cal R} f(x)$.

We provide an equivalent characterization of $\hat \partial_{\cal R} f(x)$ below.
\begin{proposition}\label{eqdif}
Let $R$ be any given retraction as defined in Definition \ref{retraction}. Then
 $v\in \hat \partial_{\cal R} f(x) $ if and only if 
  $v\in \T_x {\cal M}$ and the following holds
\begin{eqnarray}\label{pullback}
f\circ R(\eta_x) \ge f \circ R(0_x) + \langle v, \eta_x \rangle + o(\|\eta_x\|), \quad \forall \eta_x \in \T_x{\cal M}.
\end{eqnarray}
\end{proposition}

\begin{proof}
By Definition \ref{mydef}, $v\in \hat \partial_{\cal R} f(x)$ if and only if there exists $h\in C^1(B_{x,\delta})$ for some $\delta>0$ such that $f-h$ attains local minimum at $x$ on $\cal M$, and $\operatorname{grad} h(x) = v$. The latter statement is equivalent to the fact that  $f\circ R_x - h\circ R_x$ obtains local minimum at $0_x$
in $\T_x {\cal M}$. By Definition \ref{retraction} and the fact that $\cal M$ is endowed with a Riemannian metric, we have
\[
 \operatorname{grad}  (h\circ R_x)(0_x) = \operatorname{grad}  h(x) = v.
\]
This implies that $v\in \hat \partial (f\circ R_x)(0_x)$, when considering $\T_x {\cal M}$ is an Euclidean space itself. By Definition \ref{subdif-Rn}, we know that $v\in \T_x {\cal M}$ satisfies \eqref{pullback}. 
\end{proof}

According to Proposition \ref{eqdif}, we easily find that if $\bar x$ is a local minimizer of $f$ on ${\cal M}$, then $0 \in \hat \partial_{\cal R} f(\bar x)$.

\begin{definition}\label{stationary-point}
A point $x\in {\cal M}$ is called a limiting stationary point of the Rimannian optimization problem \eqref{model}, if $0\in \partial_{\cal R} f(x)$.
\end{definition}

The algorithm proposed in this paper is related to the smoothing function $\tilde f$ that is employed. It is natural that the convergence result also relates to $\tilde f$. We give the following definition for Riemannian subdifferential of $f$ associated with $\tilde f$ at $x\in \cal M$.

\begin{definition}
The subdifferential of $f$ associated with $\tilde f$ at $x\in \mathbb{R}^n$ is
\begin{eqnarray}\label{Gsf}
 G_{{\tilde f}}(x) =  
\{u\in \mathbb{R}^n\ : \ \nabla_x  \tilde f(z_k,\mu_k) \to u\quad {\rm for\ some}\ z_k \to x,\ \mu_k \downarrow 0\},
\end{eqnarray}
and the Riemannian subdifferential of $f$ associated with $\tilde f$ at $x\in \cal M$ is
\begin{eqnarray}\label{RGsf}
\quad \quad G_{{\tilde f}, {\cal R}}(x) = 
\{v \in \mathbb{R}^n &:& \operatorname{grad}  \tilde f(z_k,\mu_k) \to v 
\ {\rm for\ some}
\ z_k\in {\cal M},\ z_k \to x,\ \mu_k \downarrow 0 
\}.
\end{eqnarray}
\end{definition}

\begin{remark}
Here $u\in G_{\tilde f}(x)$  and $v\in G_{{\tilde f}, {\cal R}}(x)$ are vectors in $\mathbb{R}^n$ whose {entries are finite, i.e., they are not $\infty$ or $-\infty$.}
\end{remark}

\begin{example}
\label{example-3.1}
For the smoothing function  ${\tilde f}_{\mu}(t) = (s_{\mu}(t))^p$ of $f(t) = |t|^p$ with $0<p<1$, where $s_{\mu}(t)$ is the uniform smoothing function of $|t|$ defined in \eqref{uniform-sf}, we have
\begin{eqnarray*}
s_{\mu}'(t) = \left\{\begin{array}{ll}
                       {\rm sign}(t) & {\rm if}\ |t|\ge \frac{\mu}{2}\\
                       \frac{2t}{\mu} & {\rm if}\ |t|<\frac{\mu}{2}
                     \end{array}
              \right. 
\quad \mbox{ and } \quad
[(s_{\mu}(t))^p]' = p(s_{\mu}(t))^{p-1} s_{\mu}'(t).
\end{eqnarray*}
Here ${\rm sign}(t)=1$ if $t>0$, ${\rm sign}(t)=-1$ if $t<0$, and ${\rm sign}(t)=0$ otherwise. 
For an arbitrary real number $v\in \mathbb{R}$, and an arbitrarily chosen sequence $\mu_k \downarrow 0$,
let  $t_k = a {\mu_k}^{2-p}$ with $a = \frac{4^{p-1}v}{2p}$.
It is easy to see that
\[
\lim_{\mu_k \downarrow 0} [(s_{\mu_k}(t_k))^p]' = 2p4^{1-p} a = v.
\]
Hence $G_{\tilde f}(0) = (-\infty,\infty)$. For any point $t\ne 0$, we know that
$G_{\tilde f}(t) = p |t|^{p-1} {\rm sign}(t)$.
\end{example}

\begin{definition}
We say that $x^*\in {\cal M}$ is a stationary point of $f$ associated with $\tilde f$ on the submanifold $\cal M$,
 if
\begin{eqnarray}
\liminf_{x \to x^*,\ x\in {\cal M},\ \mu \downarrow 0} \| \operatorname{grad}  \tilde f (x,\mu)\| = 0.
\end{eqnarray}
\end{definition}

The following result is an extension of Proposition 3.4 of \cite{Zhang2} from $\mathbb{R}^n$ to the
submanifold $\cal M$.

\begin{proposition} 
For any smoothing function $\tilde f$ of $f$ as defined in Definition \ref{sf},
if $x^* \in {\cal M}$ is a local minimizer of $f$ on the submanifold $\cal M$, then
$x^*$ is a stationary point of $f$ associated with $\tilde f$ on the submanifold $\cal M$.
\end{proposition}
\begin{proof}
Note that $x^* \in {\cal M}$ is a local minimizer of $f$ on the submanifold $\cal M$.
Since minima are preserved by composition with diffeomorphisms (see, e.g., the proof of (2) $\Rightarrow$ (1) in Proposition 2.2 of \cite{Azagra}), we then know that $0_{x^*}$ is a local minimizer of $\hat f = f \circ R_{x^*}$ on the tangent space
$\T_{x^*}{\cal M}$. Hence there exists a neighborhood $B_{0_{x^*},\delta}$ of $0_{x^*}$ such that for any $\eta \in \T_{x^*} {\cal M}\cap B_{0_{x^*},\delta}$, it holds that $\hat f(0_{x^*}) \le \hat f(\eta)$.

Let us denote $\hat{f}_{\mu} = \tilde f_{\mu} \circ R_{x^*}$ for any fixed $\mu >0$. We have
\begin{eqnarray*}
 {\hat f}_{\mu}(0_{x^*}) = \tilde f(x^*,\mu) &\le& f(x^*) + \kappa \omega(\mu)\\
                                      & =& \hat f(0_{x^*}) + \kappa \omega(\mu)\\
                                      &\le &  \hat f (\eta) + \kappa \omega(\mu)\quad {\rm for\ any}\ \eta \in
                                                                           B_{0_{x^*},\delta}\\
                                      &= &  f(x) + \kappa \omega(\mu)\quad {\rm for}\ x = {R_{x^*}(\eta)}\\
                                      &\le &  \tilde f(x,\mu) + 2 \kappa \omega(\mu)\\
                                      &=& \hat f_{\mu}(\eta) + 2 \kappa \omega(\mu).
\end{eqnarray*}
Thus,
\begin{eqnarray}\label{ineq1}
\hat{f}_{\mu}(0_{x^*}) \le \hat{f}_{\mu}(\eta) + 2 \kappa \omega(\mu),\quad {\rm for \ any}\ \eta \in B_{0_{x^*},\delta}.
\end{eqnarray}
For any $\eta_z \in \T_{x^*} {\cal M} \cap B_{0_{x^*},\delta}$, we define $\eta_{\mu} = 0_{x^*} + \sqrt{\omega(\mu)}\eta_z \in \T_{x^*} {\cal M} \cap B_{0_{x^*},\delta}$  for all $\mu$ sufficiently small, and $\eta_{\mu} \to 0_{x^*}$ as $\mu \downarrow 0$. Since $\hat f_{\mu}$ is continuously differentiable on $\T_{x^*}{\cal M}$, by Taylor's expansion we have
\begin{eqnarray}\label{eq1}
 \hat f_{\mu}(0_{x^*}) 
 = \hat f_{\mu}(\eta_{\mu}) + \langle  \operatorname{grad}  \hat f_{\mu}(\eta_{\mu}),-\sqrt{\omega(\mu)} \eta_z \rangle +o(\sqrt{\omega(\mu)}\| \eta_{z}\|).
\end{eqnarray} 
Substituting \eqref{eq1} into the left hand side of \eqref{ineq1}, and replacing $\eta$ by $\eta_{\mu}$ with $\mu$ that is sufficiently small, we get
\[
\sqrt{\omega(\mu)} \langle  \operatorname{grad}   \hat f_{\mu}(\eta_{\mu}),-\eta_z \rangle + o(\sqrt{\omega(\mu)}\|\eta_z\|) \le 2 \kappa \omega(\mu).
\]
Dividing both sides of the above inequality by $\sqrt{\omega(\mu)}$, and taking the limit as $\mu \downarrow 0$, we get 
\[
  \limsup_{\mu \downarrow 0} \langle  \operatorname{grad}   \hat f_{\mu}(\eta_{\mu}),-\eta_z \rangle  \le 0,
\]
which implies that
\begin{eqnarray}\label{main-inequality}
\liminf_{\eta \to 0_{x^*},\ \eta\in \T_{x^*}{\cal M}, \ \mu \downarrow 0} \langle  \operatorname{grad}   \hat f_{\mu}(\eta),-\eta_z \rangle  \le 0.
\end{eqnarray}

Note that $\eta_z \in \T_{x^*}{\cal M} \cap B_{0_{x^*},\delta}$ can be chosen arbitrarily. Let $\cal M$ be a $d$-dimensional submanifold. We can choose $E: \mathbb{R}^n \to \T_{x^*}{\cal M}$ to be a linear bijection such that $\{E(e_i)\}_{i=1}^d$ is an orthonormal basis of $\T_{x^*} {\cal M}$, where $e_i$ is the $i$-th unit vector (see, e.g., Section 2 of \cite{Yang}). Then
  \begin{eqnarray}\label{gradfhat}
  \operatorname {grad}  {\hat f}_{\mu}(\eta) = \sum_{i=1}^d \lambda_i^{\mu} E(e_i),
  \end{eqnarray}
  for some $\lambda_i^{\mu} \in \mathbb{R}$. Let us choose
  \begin{eqnarray*}
  \eta_z^{(i,1)} = \epsilon_i E(e_i),\ \eta_z^{(i,2)} = -\epsilon_i E(e_i),\quad \mbox{for}\ i=1,2,\ldots,d,
  \end{eqnarray*}
  where $\epsilon_i >0$ is a sufficiently small  constant such that $\eta_z^{(i,1)}, \eta_z^{(i,2)} \in B_{0_{x^*},\delta}. $
  Substituting ${\rm grad} \hat f_{\mu}(\eta)$ in \eqref{main-inequality} by \eqref{gradfhat}, and substituting
   $\eta_z$  in \eqref{main-inequality} by $\eta_z^{(i,1)}$ and $\eta_z^{(i,2)}$, respectively, we obtain
\[
   \liminf_{\mu \downarrow 0} -\epsilon_i \lambda_i^{\mu} \ge 0,\quad \mbox{and}\quad
   \liminf_{\mu\downarrow 0} \epsilon_i \lambda_i^{\mu} \ge 0.
\]
The above two inequalities indicate
\[
     \lim_{\mu \downarrow 0} \lambda_i^{\mu} =0.
\]
Since $i=1,2,\ldots,d$ can be chosen arbitrarily, the above equality holds for each $i$. Hence, we get
\[
   \liminf_{\eta\to 0_{x^*},\ \eta\in T_{x^*}{\cal M},\ \mu \downarrow 0} \| \operatorname{grad}   \hat f_{\mu}(\eta)\|
  = \lim_{\mu \downarrow 0} \|\sum_{i=1}^d \lambda_i^{\mu} E(e_i)\| =0.
\]
That is,
\[
  \liminf_{x\to x^*,\ x\in {\cal M},\ \mu \downarrow 0}  \| \operatorname{ grad} \tilde f(x,\mu)\| =0,
\]
and hence $x^*$ is a stationary point of $f$ associated with $\tilde f$ on $\cal M$ as desired.
\end{proof}

We will show later that any accumulation point of our proposed RSSD method is a stationary point of $f$ associated with $\tilde f$ on $\cal M$.

\subsection{Riemannian gradient sub-consistency}

Now we define the Riemannian gradient sub-consistency of $\tilde f$ at $x\in \cal M$, which makes connection between the Riemannian subdifferential $G_{\tilde f, {\cal R}}(x)$ associated with $\tilde f$ and the Riemannian limiting subdifferential $\partial_{\cal R} f(x)$. The Riemannian gradient sub-consistency will be essential to show that any accumulation point of the RSSD method developed in this paper is also a Riemannian limiting stationary point.

\begin{definition}
\label{weakgradient}
A smoothing function $\tilde f$ of the function $f$ is said to satisfy the  gradient sub-consistency at $x\in {\mathbb{R}^n}$ if
\begin{eqnarray}\label{weak1}
G_{\tilde f}(x) \subseteq \partial f(x),
\end{eqnarray}
and $\tilde f$ is said to satisfy the Riemannian gradient sub-consistency at $x\in {\cal M}$ if
\begin{eqnarray}\label{Rweakgradient}
G_{\tilde f, {\cal R}}(x) \subseteq \partial_{\cal R} f(x).
\end{eqnarray}
We say that $\tilde f$ satisfies the gradient sub-consistency on $\mathbb{R}^n$ if \eqref{weak1} holds for any $x\in \mathbb{R}^n$, and that $\tilde f$ satisfies the Riemannian gradient sub-consistency on $\cal M$ if \eqref{Rweakgradient} holds for any $x\in {\cal M}$.
\end{definition}

If the inclusion is substituted by the equality in \eqref{weak1} for any $x\in \mathbb{R}^n$, then $\tilde f$ satisfies the gradient consistency on $\mathbb{R}^n$. Clearly, the gradient consistency indicates the gradient sub-consistency. The gradient consistency on $\mathbb{R}^n$ has been well studied in smoothing methods for nonsmooth optimization. For nonsmooth but Lipschitz function $f$, it has been shown that the gradient consistency property holds for various smoothing functions in many real applications \cite{Burke,Burke2,Chen,XuYe2015,Zhang}.   

The following proposition demonstrates that if the gradient sub-consistency of $\tilde f$ in $\mathbb{R}^n$ holds, then the Riemannian gradient sub-consistency of $\tilde f$ holds on $\cal M$, provided that $f$ is locally Lipschitz.

\begin{proposition}\label{propostion-gradient-consistency}
Let $f$ be a locally Lipschitz function with $\tilde f$ being a smoothing function of $f$. If the gradient sub-consistency of $\tilde f$ holds on $\mathbb{R}^n$, then the Riemannian gradient sub-consistency on $\cal M$ holds. 
\end{proposition}

\begin{proof} 
For any $x\in {\cal M}$, let $v\in G_{{\tilde f},{\cal R}}(x)$. Note that $G_{\tilde f}(x) \subseteq \partial f(x)$ is bounded if $f$ is a locally Lipschitz function. Then there exist subsequences $x_{\mu_k} \in {\cal M}$, $x_{\mu_k}\to x$, $\mu_k \downarrow 0$ as $k\to \infty$, and a vector $u\in G_{\tilde f}(x)$ such that
\begin{eqnarray}\label{limit-1}
u = \lim_{x_{\mu_k} \to x, \ x_{\mu_k}\in {\cal M},\ \mu_k \downarrow 0} \nabla_x \tilde f(x_{\mu_k},\mu_k),
\end{eqnarray}
and
\begin{eqnarray}\label{limit-2}
v & = & \lim_{x_{\mu_k} \to x,\ x_{\mu_k}\in {\cal M},\ \mu_k \downarrow 0} \operatorname{grad} \tilde f(x_{\mu_k}, \mu_k) \nonumber\\
&=& \lim_{x_{\mu_k} \to x,\ x_{\mu_k}\in {\cal M},\ \mu_k \downarrow 0} \operatorname{Proj}_{\T_{x_{\mu_k}}{\cal M}} \nabla_x \tilde f(x_{\mu_k},\mu_k),
\nonumber\\
&=& \operatorname{Proj}_{\T_x{\cal M}} u. 
\end{eqnarray}
The last equality holds because
\begin{align*}
 & \|\operatorname{Proj}_{\T_{x_{\mu_k}}{\cal M} } \nabla_x \tilde f(x_{\mu_k},\mu_k) - \operatorname{Proj}_{\T_x{\cal M}} u\|  \\
  \le &\  \|\operatorname{Proj}_{\T_{x_{\mu_k}}{\cal M} } \nabla_x \tilde f(x_{\mu_k},\mu_k) - \operatorname{Proj}_{\T_{x_{\mu_k}}{\cal M}} u \| + \|\operatorname{Proj}_{\T_{x_{\mu_k}}{\cal M}}u - \operatorname{Proj}_{\T_{x}{\cal M}}u \| \\
   \le &\ \|\nabla_x \tilde f(x_{\mu_k},\mu_k) - u\| + \|\operatorname{Proj}_{\T_{x_{\mu_k}}{\cal M}}u - \operatorname{Proj}_{\T_{x}{\cal M}}u \|  \\
 \to &\  0, 
\end{align*}
as\  $x_{\mu_k} \to x,\ x_{\mu_k}\in {\cal M},\ \mu_k \downarrow 0.$
Here the second inequality comes from the fact that $\operatorname{Proj}_{\T_{x_{\mu_k}}{\cal M}}$ is nonexpansive. Moreover, $\|\nabla_x \tilde f(x_{\mu_k},\mu_k) - u\|\to 0$ by \eqref{limit-1}, and $\|\operatorname{Proj}_{\T_{x_{\mu_k}}{\cal M}}u - \operatorname{Proj}_{\T_{x}{\cal M}}u \|\to 0$ because  
$S(x) := \T_x \cal M$ is continuous 
 and convex-valued (i.e.,  $S(x)$ is a convex set for each fixed $x$), and ${\operatorname{Proj}_{S(x)}}$ is continuous according to Example 5.57 of \cite{Rockafellar}.

Since the gradient sub-consistency $ G_{\tilde f} \subseteq \partial f(x)$ holds, we  know that $u\in \partial f(x)$. By the definition of the limiting subdifferential of $f$ on $\mathbb{R}^n$,
\[
\exists \ u_{\ell} \in \hat \partial f(x_{\ell}),\ (x_{\ell}, f(x_{\ell})) \to (x,f(x))\ \mbox{such that}\ \lim_{\ell \to \infty}  u_{\ell} = u.
\]
By the characterization of the Riemannian Fr\'{e}chet subdifferential 
in \eqref{hat_partial_R_2}, we have
\[
v_{\ell}= \operatorname{Proj}_{\T_{x_{\ell}}{\cal M}}u_{\ell}  \in \hat \partial_{\cal R} f(x_{\ell}),
\]
and using the same arguments of proving \eqref{limit-2}, we have  
\begin{eqnarray*}
\lim_{\ell \to \infty} v_{\ell} = \lim_{\ell \to \infty} \operatorname{Proj}_{\T_{x_{\ell}}{\cal M}}  u_{\ell}  
= \operatorname{Proj}_{\T_x{\cal M}}  u = v.
\end{eqnarray*}
This implies $v\in \partial_{\cal R} f(x)$,  and hence the Riemannian gradient sub-consistency holds.
\end{proof}

 For non-Lipschitz functions, we first use the smoothing function $s_{\mu}(t)$  of $|t|^p$ to illustrate that the gradient consistency on $\mathbb{R}^n$ holds. 
It is known that   $\partial f(0)=  (-\infty,\infty)$, and  $\partial f(t) = p |t|^{p-1} {\rm sign}(t)$.  This, combined with Example \ref{example-3.1}, yields that
$$ G_{\tilde f}(0) = (-\infty,\infty) = \partial f(0),$$
and
for any point $t\ne 0$, 
$$G_{\tilde f}(t) = p |t|^{p-1} {\rm sign}(t) = \partial f(t).$$
Thus the smoothing function $\tilde f$ of the non-Lipschitz function $f = |t|^p$ satisfies the gradient consistency  on $\mathbb{R}^n$.

Furthermore, we consider a class of non-Lipschitz optimization on submanifold $\cal M$ as follows
\begin{eqnarray}\label{non-Lipschitz-example}
\min_{x\in {\cal M}} \  f(x):=  \hat f(x) +  \lambda \|Bx\|_p^p,
\end{eqnarray}
where $\hat f$ is a smooth function, $\cal M$ is a submanifold, $B\in \mathbb{R}^{m\times n}$ is a given matrix of full column rank,  and  $p\in (0,1)$, and $\lambda>0$ are given constants. Many applications can be formulated in the form of \eqref{non-Lipschitz-example}, such as finding the sparsest vector in a subspace, and the sparsely-used orthogonal complete dictionary learning that will be discussed later in Section \ref{sec:numerical}. Let $\tilde s_{\mu}(t)$ be a smoothing function of $|t|$ satisfying Definition \ref{sf}. Then the function 
\begin{eqnarray}\label{smoothing-example}
\tilde f(x,\mu) = \hat f(x) + \lambda\sum_{i=1}^m [\tilde s_{\mu}((Bx)_i) ]^p 
\end{eqnarray}
is a smoothing function of $f$ defined in \eqref{non-Lipschitz-example}. We then have the following proposition. 

\begin{proposition} 
\label{example-gradient-sub-consistency-lp} The smoothing function $\tilde f$ that is constructed  in \eqref{smoothing-example} for the non-Lipschitz objective function $f$  in \eqref{non-Lipschitz-example} satisfies  the gradient sub-consistency on $\mathbb{R}^n$, and  the Riemannian gradient sub-consistency on the submanifold $\cal M$. 
\end{proposition}
\begin{proof}
For any $x\in {\mathbb{R}^n}$, let us denote the index sets
\[
 I_1 := \{i\mid (Bx)_i \ne 0\},\quad\mbox{and}\quad I_2:= \{i\mid (Bx)_i =0\},
\]
and correspondingly for any $z\in \mathbb{R}^n$, define
\begin{eqnarray*}
 f_1(z) := \lambda \sum_{i\in I_1} |(Bz)_i|^p,\quad \mbox{and} 
\quad f_2(z):=\lambda \sum_{i\in I_2} |(Bz)_i|^p,\\
 \tilde f_1(z,\mu):= \lambda\sum_{i\in I_1} [\tilde s_{\mu}((Bz)_i)],
\quad \mbox{and}\quad
\tilde f_2(z,\mu):=\lambda  \sum_{i\in I_2} [\tilde s_{\mu}((Bz)_i)].
\end{eqnarray*}
Clearly 
\[
\lambda \|Bz\|_p^p = f_1(z)+f_2(z),\quad \mbox{and}\quad \tilde f(z,\mu) = \hat f(z) + \tilde f_1(z,\mu) + \tilde f_2(z,\mu).
\]

For any $u\in G_{\tilde f}(x)$, we know that there exist sequence $z_k \to x$, and $\mu_k \downarrow 0$ as $k\to \infty$ such that 
\begin{eqnarray}\label{u}
u = \lim_{z_k \to x,\ \mu_k \downarrow 0} \nabla_x \tilde f(z_k,\mu_k).
\end{eqnarray}
It is clear that 
\begin{eqnarray}\label{gra-tf-exp}
\nabla_x \tilde f(z_k,\mu_k) = \nabla \hat f(z_k) + \nabla_x \tilde f_1(z_k,\mu_k) + \nabla_x \tilde f_2(z_k,\mu_k),
\end{eqnarray}
and 
\begin{eqnarray}\label{gra-tf-exp-2}
\lim_{k\to \infty} \nabla \hat f(z_k) = \nabla \hat f(x) \quad \mbox{and}\quad \lim_{k\to \infty} \nabla_x \tilde f_1(z_k,\mu_k) = \nabla f_1(x). 
\end{eqnarray}

By direct computation,
\begin{eqnarray} \label{gra-tf2}
\nabla_x \tilde f_2(z_k,\mu_k)=
        \sum_{i\in I_2}\lambda p\left(\tilde s_{\mu_k}((Bz_k)_i)\right)^{p-1} [\tilde s_{\mu_k}((Bz_k)_i)]' B_{i.}^\top =  B_{I_2}^\top Y_{k}.
\end{eqnarray}
Here $B_{i.}$ is the $i$-th row of $B$, $B_{I_2}$ is the submatrix of $B$ defined by  $B_{I_2} = (B_{i.})_{i\in I_2}$, and   
\[
Y_{k}:= Y_{k}(z_k,\mu_k)=\lambda p \left((\tilde s_{\mu_k}((Bz_{k})_i))^{p-1}  [\tilde s_{\mu_k}((Bz_k)_i)]'\right)_{i\in I_2} \in \mathbb{R}^{|I_2|},
\]
with $|I_2|$ being the cardinality of the index set $I_2$. 
Let $N(C)$ be the null space of the matrix $C$ and $N(C)^{\perp}$ be its orthogonal complement.
It is known that $Y_{k}$ can be uniquely written as 
\begin{eqnarray}
\label{dec-Y}Y_{k} = Y_{k}^1 + Y_{k}^2, 
\quad \mbox{where}\ Y_{k}^1 \in N(B_{I_2}^\top),\ Y_{k}^2\in N(B_{I_2}^\top)^{\perp} .
\end{eqnarray}

We claim that $\{Y_{k}^2\}$ is bounded along with $z_k \to x$, $\mu_k \downarrow 0$ as $k\to \infty$.
Otherwise, there exists an infinite subsequence $K_1\subseteq\{1,2,\ldots\}$ such that 
 $$\lim_{k\to \infty,\ k\in K_1} \|Y_{k}^2\|= \infty.$$ 
Because 
\[
\frac{Y_{k}^2}{\|Y_{k}^2\|} \in N(B_{I_2}^\top)^{\perp}, \quad \mbox{and}\ \left\{\frac{Y_{k}^2}{\|Y_{k}^2\|}\right\} \ \mbox{is bounded}, 
\]
there exists an infinite subsequence $K_2 \subseteq K_1$ such that 
\begin{eqnarray}\label{barY}
\lim_{ k\to \infty,\ k\in K_2} \frac{Y_{k}^2}{\|Y_{k}^2\|} = \bar Y \in N(B_{I_2}^\top)^{\perp} , \quad \mbox{with}\ \|\bar Y\| = 1.
\end{eqnarray}
This, together with \eqref{gra-tf2} implies that $B_{I_2}^\top \bar Y \ne 0,$
and 
\begin{eqnarray*}
\lim_{k\to \infty,\ k\in K_2} \|\nabla_x \tilde f_2(z_k,\mu_k) \|
&=& 
  \lim_{k\to \infty,\ k\in K_2}\|B_{I_2}^\top \left(Y_{k}^1 + Y_{k}^2 \right)\|\\
&=&
  \lim_{k\to \infty,\ k\in K_2}  \left\|B_{I_2}^\top \left(\|Y_{k}^2\| \frac{Y_{k}^2}{\|Y_{k}^2\|} \right)   \right\| = \infty.
\end{eqnarray*} 
Hence, by using \eqref{gra-tf-exp} and \eqref{gra-tf-exp-2}, we find
$
 \|\nabla_x \tilde f(z_k,\mu_k)\|   \to \infty 
$
as $k\to \infty, k\in K_2$,
which contradicts to \eqref{u} that $u \in \mathbb{R}^n$ cannot have components tending to infinity.

From the boundedness of $\{Y_{k}^2\}$, we know that there exists an infinite subsequence
 $K_3 \subseteq \{1,2,\ldots\}$ such that 
\begin{eqnarray*}
\lim_{k\to \infty,\ k\in K_3} Y_{k}^2 = \hat Y \in \mathbb{R}^{I_2}.
\end{eqnarray*}
Hence 
\begin{eqnarray}\label{u2}
u = \lim_{z_k \to x,\ \mu_k \downarrow 0} \nabla_x \tilde f(z_k,\mu_k)
    = \nabla \hat f(x) + \nabla f_1(x) +  B_{I_2}^\top \hat Y.
\end{eqnarray}
Let us define the function
\[
h(z) = \hat f(z) + f_1(z) +   \sum_{i\in I_2} \hat Y_i (Bz)_i. 
\]
Note that for any $\nu\in R$ , there exists  some $\delta>0$ such that 
\[
|t|^p > \nu t\quad \mbox{for any}\ |t|\le\delta.
\]
We can easily find that there exists a neighborhood $B_{x,\bar \delta}$ of $x$ such that  for any $z\in B_{x,\bar \delta}$, $h(z)\le  f(z)$, and $h(x) = f(x)$. Thus by Proposition \ref{prop2.2}, we have
\[
\nabla h(x) = \nabla \hat f(x) + \nabla f_1(x) +  B_{I_2}^\top \hat Y \in \hat \partial f(x) \subseteq \partial f(x). 
\]
This, combining with \eqref{u2} yields $u\in \partial f(x)$. Since both $x\in \mathbb{R}^n$ and $u\in G_{\tilde f}(x)$ are arbitrary, we get that $\tilde f$ defined in \eqref{smoothing-example} satisfies the gradient sub-consistency on $\mathbb{R}^n$. 

Below we show that $\tilde f$ also satisfies the Riemannian gradient sub-consistency on the submanifold $\cal M$.

For any $x\in \cal M$, let $v\in G_{\tilde f, \cal R}(x)$. Then there exists infinite sequence $z_k\in \cal M$, $z_k\to x$, $\mu_k \downarrow 0$ as $k\to \infty$ such that 
\begin{eqnarray}\label{v}
v &=& \lim_{z_k \to x, \ z_k \in {\cal M},\ \mu_k \downarrow 0} \operatorname{grad} \tilde f(z_k,\mu_k) \nonumber\\
&=& \lim_{z_k\to x,\ z_k\in {\cal M}, \ \mu_k \downarrow 0} {\rm Proj}_{\T_{z_k}{\cal M}} \nabla_x \tilde f(z_k,\mu_k).
\end{eqnarray}

If $\{\nabla_x \tilde f(z_k,\mu_k)\}$ is bounded, noting that $\tilde f$ satisfies the gradient sub-consistency on $\mathbb{R}^n$, and following the similar arguments in the proof of Proposition \ref{propostion-gradient-consistency}, we can show that $v\in \partial_{\cal R} f(x)$.

Otherwise, there exists an infinite subsequence $K \subseteq \{1,2,\ldots\}$ such that 
\[
\{\| \nabla_x \tilde f(z_k,\mu_k)\|\}_{k\in K} \to \infty,
\]
which indicates that $\{\|\nabla_x \tilde f_2(z_k,\mu_k)\|\}_{k\in K} \to \infty$
by noting \eqref{gra-tf-exp} and \eqref{gra-tf-exp-2}. By \eqref{gra-tf2} and \eqref{dec-Y}, we know
\begin{eqnarray}\label{gra-tf2-new}
\nabla_x \tilde f_2(z_k,\mu_k) =  B_{I_2}^\top Y_{k}^2,\quad \mbox{where}\ Y_{k}^2 \in N(B_{I_2}^\top)^{\perp}.
\end{eqnarray}
Hence 
\begin{eqnarray}\label{tend-inf}
 \{\|B_{I_2}^\top Y_{k}^2\|\}_{k\in K}  \to \infty,\quad \mbox{and}\ \{\|Y_k^2\|\}_{k\in K} \to \infty.
\end{eqnarray}
For any sequences $g_k^1, g_k^2 \in \mathbb{R}^n$, it is easy to see that
\begin{align*}
& \left\|\operatorname{Proj}_{\T_{z_k}\cal M} g_k^2\right\| - \left\|\operatorname{Proj}_{\T_{z_k}\cal M} (g_k^1+g_k^2)\right\| \\
\le & \left\|\operatorname{Proj}_{\T_{z_k}\cal M}  (g_k^1+g_k^2)  - \operatorname{Proj}_{\T_{z_k}\cal M} g_k^2\right\|
\le  \left\|g_k^1\right\|,
\end{align*} 
which implies 
\begin{eqnarray}\label{gk12}
\left\|\operatorname{Proj}_{\T_{z_k}\cal M} g_k^2 \right\| 
\le \left\|\operatorname{Proj}_{\T_{z_k}\cal M} (g_k^1 + g_k^2)  \right\| + \left\| g_k^1 \right\|.
\end{eqnarray}
By substituting $g_k^1 = \nabla \hat f(z_k) + \nabla_x \tilde f_1(z_k,\mu_k)$ and $g_k^2 = \nabla_x \tilde f_2(z_k,\mu_k)$ into \eqref{gk12}, we have 
\[
\left\|\operatorname{Proj}_{\T_{z_k}\cal M} \nabla_x \tilde f_2(z_k,\mu_k)\right\| \le \left\|\operatorname{Proj}_{\T_{z_k}\cal M}  \nabla_x \tilde f(z_k,\mu_k)\right\| + \left\|\nabla \hat f(z_k) + \nabla_x \tilde f_1(z_k,\mu_k)\right\|.
\]
The two terms on the right-hand side of the above inequality are bounded by noting \eqref{v} and \eqref{gra-tf-exp-2}. Thus 
\begin{eqnarray}\label{proj-bounded}
\left\{\left\|\operatorname{Proj}_{\T_{z_k}\cal M} \nabla_x \tilde f_2(z_k,\mu_k)\right\|\right\}\  \mbox{is bounded.}
\end{eqnarray}
Using \eqref{tend-inf}, we may assume without loss of generality that
\[
\lim_{k\to \infty,\ k\in K} \frac{Y_k^2}{\|Y_k^2\|}=\bar Y \ne 0.
\]
We can write  
\begin{eqnarray}\label{decom1}
B_{I_2}^\top  \frac{Y_k^2}{\|Y_k^2\|} = d_k^1 + d_k^2,\quad \mbox{where}\ d_k^1\in \T_{z_k}{\cal M}, d_k^2 \in (\T_{z_k} {\cal M})^{\perp};\\
\label{decom2}
\nabla \hat f(z_k) + \nabla_x \tilde f_1(z_k,\mu_k) = a_k^1 + a_k^2, \quad \mbox{where}\ a_k^1\in \T_{z_k}{\cal M}, a_k^2 \in (\T_{z_k} {\cal M})^{\perp}.
\end{eqnarray}
Here $(\T_{z_k} {\cal M})^{\perp}$ is the orthogonal complement of $\T_{z_k}{\cal M}$.

For any scalar $\alpha>0$ and $g_k\in \mathbb{R}^n$, it is not difficult to show that 
\begin{eqnarray}\label{proj-property-1}
\operatorname{Proj}_{\T_{z_k}{\cal M}} \alpha g_k = \alpha \operatorname{Proj}_{\T_{z_k}{\cal M}} g_k. 
\end{eqnarray}
Thus 
\begin{align*}
   & \operatorname{Proj}_{\T_{z_k}{\cal M}} \nabla_x \tilde f_2(z_k,\mu_k) \\
= & \operatorname{Proj}_{\T_{z_k}{\cal M}} B_{I_2}^\top Y_k^2 = \operatorname{Proj}_{\T_{z_k}{\cal M}} \|Y_k^2\| B_{I_2}^\top \frac{ Y_k^2}{\|Y_k^2\|}\\
=& \|Y_k^2\|\operatorname{Proj}_{\T_{z_k}{\cal M}} B_{I_2}^\top \frac{ Y_k^2}{\|Y_k^2\|}= \|Y_k^2\|\operatorname{Proj}_{\T_{z_k}{\cal M}} (d_k^1 + d_k^2)= \|Y_k^2\| d_k^1,
\end{align*}
where the third equality employs \eqref{proj-property-1}. In view of \eqref{tend-inf} and \eqref{proj-bounded}, we get
\[
\lim_{k\to \infty,\ k\in K} d_k^1 = 0.
\]
By using \eqref{decom1} and \eqref{decom2}, we get 
\begin{eqnarray*}
\operatorname{Proj}_{\T_{z_k}{\cal M}} \nabla_x \tilde f(z_k,\mu_k) 
&=& \operatorname{Proj}_{\T_{z_k}{\cal M}} (\nabla \hat f(z_k) + \nabla_x \tilde f_1(z_k,\mu_k) + B_{I_2}^\top Y_k^2)\\
&=& \operatorname{Proj}_{\T_{z_k}{\cal M}} (a_k^1 + a_k^2+ d_k^1 + d_k^2)\\
&=& \operatorname{Proj}_{\T_{z_k}{\cal M}}(a_k^1 + d_k^1)= a_k^1 + d_k^1.
\end{eqnarray*}
Consequently, 
\begin{eqnarray*}
v &=& \lim_{z_k\to x,\ z_k\in {\cal M},\ \mu_k \downarrow 0} \operatorname{Proj}_{T_{z_k}{\cal M}} \nabla_x \tilde f(z_k,\mu_k) \\
&=& \lim_{k\to \infty,\ k\in K} (a_k^1 + d_k^1)= \lim_{k\to \infty,\ k\in K} a_k^1\\
&=& \lim_{z_k\to x,\ z_k\in {\cal M},\ \mu_k \downarrow 0} \operatorname{Proj}_{T_{z_k}{\cal M}} (\nabla \hat f(x) + \nabla f_1(x)).
\end{eqnarray*}
We now define function $\bar{h}(z) = \hat f(z) + f_1(z)$. It is then easy to check that there exists a neighborhood $B_{x,\delta}$ 
 for some $\delta>0$ such that
$\bar h(z)\le f(z)$ with $\bar h(x) = f(x)$, and $\nabla \bar h(x) = \nabla \hat f(x) + \nabla f_1(x)$. Then by Proposition \ref{prop2.2}, $\nabla \bar h(x)  \in \hat \partial f(x)$. 
Hence 
\begin{eqnarray*}
v= \operatorname{Proj}_{T_{z_k}{\cal M}}(\nabla \hat f(x)+\nabla f_1(x)) \in \hat \partial_{\cal R}f(x) \subseteq \partial_{\cal R}f(x).
\end{eqnarray*}
Therefore, $\tilde f$ satisfies the Riemannian gradient sub-consistency as desired.
\end{proof}

\section{Riemannian smoothing steepest descent method}\label{sec:Riemannian RSSD}

In this section, we present our RSSD method, which is detailed in Algorithm \ref{RSSD}.
\begin{algorithm}
\caption{Riemannian smoothing steepest descent method (RSSD) for solving \eqref{model}}\label{RSSD}
\begin{algorithmic}[1]
\STATE{{\bf Input:} $x_0 \in {\cal M}$, $\delta_{opt}\ge 0$, $\delta_0 >0$, $\mu_{opt}\ge 0$, $\mu_0 >0$, $\beta\in (0,1)$, $\bar \alpha>0$, $\theta_{\delta}\in (0,1)$, $\theta_{\mu}\in (0,1)$.} 
\FOR{$\ell = 0,1,2,\ldots$}
\STATE{Compute $\eta_{\ell} = - \operatorname{grad} \tilde f(x_{\ell},\mu_{\ell}).$}
\IF{$\|\eta_{\ell}\| \le \delta_{opt}$ and $\mu_{\ell} \le \mu_{opt}$}
	\STATE{return}
\ELSIF{$\|\eta_{\ell}\| \le \delta_{\ell}$}
	\STATE{$\mu_{\ell+1} := \theta_{\mu} \mu_{\ell}, \delta_{\ell+1} :=\theta_{\delta} \delta_{\ell}$,}
	\STATE{$x_{\ell +1}:= x_{\ell}$.}
\ELSE
	\STATE{$\mu_{\ell + 1} = \mu_{\ell}, \delta_{\ell+1} = \delta_{\ell}$.}
	\STATE{Find $t_{\ell}:= \beta^m \bar \alpha$  where $m$ is the smallest integer such that
\begin{equation}\label{Armijo}
\tilde f(R_{x_{\ell}}(\beta^m \bar \alpha \eta_{\ell}),\mu_{\ell}) \le
\tilde f(x_{\ell},\mu_{\ell}) - \sigma \beta^m \bar \alpha\|\operatorname{grad} \tilde f(x_{\ell},\mu_{\ell})\|^2.
\end{equation}}
\STATE{Set $x_{\ell+1} := R_{x_{\ell}}(t_{\ell} \eta_{\ell}).$}
\ENDIF
\ENDFOR
\end{algorithmic}
\end{algorithm}

A few remarks for Algorithm \ref{RSSD} are in demand. First, the line search \eqref{Armijo} is well defined and $t_{\ell}$ can be found in finite trials. Note that for fixed $\mu_{\ell}$, $\tilde f(\cdot,\mu_{\ell})$ is continuously differentiable. Clearly, we have
\[
  \lim_{t\downarrow 0} \frac{\tilde f_{\mu_{\ell}} \circ R_{x_{\ell}} (t \eta_{\ell})-
  \tilde f_{\mu_{\ell}} \circ R_{x_{\ell}}(0_{x_{\ell}})}{t}= (\tilde f_{\mu_{\ell}} \circ R_{x_{\ell}})'(0_{x_{\ell}},\eta_{\ell}) = \langle \operatorname{grad} \tilde f(x_{\ell}, \mu_{\ell}), \eta_{\ell} \rangle.
\]
Note that $\eta_{\ell} = -\operatorname{grad} \tilde f(x_l,\mu_l)$. Thus there exists $\alpha >0$ such that for all $t\in (0,\alpha)$,
\[
  \tilde f_{\mu_{\ell}} \circ R_{x_{\ell}}(t \eta_{\ell}) \le \tilde f_{\mu_{\ell}} \circ R_{x_{\ell}} (0_{x_{\ell
  }}) - t \sigma \|\operatorname{grad} \tilde f(x_{\ell},\mu_{\ell})\|^2.
\]
This guarantees that the line search step \eqref{Armijo} is well defined.

The convergence result of Algorithm \ref{RSSD} requires the following assumption.

\begin{assumption}\label{assump}
For any fixed $\bar \mu>0$ and any given vector $\bar x \in {\cal M}$, the level set ${\cal L}_{\bar x, \bar \mu} = \{x\in {\cal M}\mid \tilde f(x,\bar \mu) \le \tilde f(\bar x, \bar \mu) \}$ is compact.
\end{assumption}

It is easy to see that this assumption holds if $\cal M$ is a sphere or the Stiefel manifold.

\begin{proposition}\label{pro1-converge}
Assume Assumption \ref{assump} holds. Let $K = \{\ell \mid \|\eta_{\ell}\| \le \delta_{\ell}\}$ and  $\{x_{\ell}\}$ be an infinite sequence
generated by Algorithm \ref{RSSD} with $\delta_{opt} = \mu_{opt} =0$.
Then $ K$ is an infinite set and
\begin{eqnarray}\label{converge-1}
\lim_{\ell \to \infty,\ \ell \in {K}} \delta_{\ell} =0\quad \mbox{and}\quad
\lim_{\ell \to \infty,\ \ell \in { K}} \mu_{\ell} = 0.
\end{eqnarray}
\end{proposition}

\begin{proof}
Suppose on the contrary that ${K}$ is a finite set. This means there exists
$\bar {\ell}$ such that for all $\ell \ge \bar {\ell}$,
\[
\delta_{\ell} \equiv \delta_{\bar \ell},\quad \mu_{\ell} \equiv \mu_{\bar \ell}, \quad
\mbox{and}\quad
\|\eta_{\ell}\| > \delta_{\bar \ell}.
\]
Therefore, for $\ell  \ge \bar {\ell}$, we have $x_{\ell+1} = R_{x_{\ell}} (t_{\ell} \eta_{\ell})$, where $t_{\ell}$ is obtained by using the line search \eqref{Armijo} with fixed $\mu_{\bar \ell}$. Using Assumption \ref{assump} and Corollary 4.3.2 of \cite{Absilbook}, we obtain
\[
\lim_{\ell \to \infty} \|\eta_{\ell}\| = \lim_{\ell \to \infty}
  \|\operatorname{grad} \tilde f(x_{\ell},\mu_{\ell}) \| =  \lim_{\ell \to \infty}
  \|\operatorname{grad} \tilde f(x_{\ell},\mu_{\bar \ell}) \|=0,
\]
which contradicts to $\|\eta_{\ell}\| > \delta_{\bar \ell}$ for all $\ell \ge \bar \ell$.
Therefore, $ K$ is an infinite set.

Note that for each $\ell \in {K}$, we have
\[
\mu_{\ell +1} = \theta_{\mu} \mu_{\ell}\quad \mbox{and}\quad \delta_{\ell+1} = \theta_{\delta} \delta_{\ell}
\]
with decaying factors $\theta_{\mu} \in (0,1)$ and $\theta_{\delta}\in (0,1)$.
This, together with $ K$ being an infinite set, yields \eqref{converge-1} as desired.
\end{proof}

\begin{theorem}\label{convergence-theorem}
 Assume Assumption \ref{assump} holds. Let $K = \{\ell\mid\|\eta_{\ell}\| \le \delta_{\ell}\}$ and  $\{x_{\ell}\}$ be an infinite sequence
generated by Algorithm \ref{RSSD} with $\delta_{opt} = \mu_{opt} =0$. Then the following statements hold.
\begin{itemize}
\item[(i)] Any accumulation point of $\{x_{\ell}\}_{\ell \in K}$ is a stationary point of \eqref{model} associated with $\tilde f$ on the submanifold $\cal M$.
\item[(ii)] In addition, if $\tilde f$ satisfies the Riemannian gradient sub-consistency, then any accumulation point of $\{x_{\ell}\}_{\ell \in K}$ is a Riemannian limiting stationary point of \eqref{model}.
\end{itemize}
\end{theorem}

\begin{proof}
By Algorithm \ref{RSSD} and Proposition \ref{pro1-converge}, we have
\begin{eqnarray*}
\lim_{\ell \to \infty,\ \ell \in {K}} \|\operatorname{grad}\tilde f(x_{\ell},\mu_{\ell})\| =
\lim_{\ell \to \infty,\ \ell \in { K}} \|\eta_{\ell}\| \le \lim_{\ell \to \infty,\ \ell \in { K}} \delta_{\ell} = 0.
\end{eqnarray*}

Let $x^*$ be any accumulation point of $\{x_{\ell}\}_{\ell \in { K}}$ with $\check{ K}$ being a subsequence of $K$ such that
$\lim_{{\ell}\to \infty,\ \ell\in \check{K}} x_{\ell} = x^*$.
Thus
$$\liminf_{x\to x^*,\ x\in {\cal M},\ \mu \downarrow 0} \|{\rm grad} \tilde f(x,\mu)\| = 0,\ \mbox{and}\ 0 \in G_{\tilde f, {\cal R}}(x^*).$$ Hence $x^*$ is a stationary point of \eqref{model} associated with $\tilde f$ on the submanifold $\cal M$. That is, statement (i) holds.

In addition, if $\tilde f$ satisfies the Riemannian gradient sub-consistency, then 
we know 
$
G_{\tilde f,{\cal R}}(x^*) \subseteq \partial_{\cal R} f(x^*).
$ 
Thus we find
$
0 \in \partial_{\cal R} f(x^*).
$
Hence $x^*$ is a Riemannian limiting stationary point of \eqref{model}. Consequently statement (ii) holds.
\end{proof}

\section{Numerical experiments}\label{sec:numerical}

In this section, we apply our RSSD method (Algorithm \ref{RSSD}) to solve two problems: finding a sparse vector in a subspace (FSV), and the sparsely-used orthogonal complete dictionary learning problem (ODL).  

\subsection{Finding a sparse vector in a subspace} 

The FSV problem seeks the sparsest vector in an $n$-dimensional linear subspace $W\subset\mathbb{R}^m$ ($m>n$). This problem has been studied recently and it finds interesting applications and connection with sparse
dictionary learning, sparse PCA, and many other problems in signal processing and machine learning \cite{Qu-FSV-TIT,Qu-survey}. This problem is also known as dual principal component pursuit and finds applications in robust subspace recovery \cite{Tsakiris-Vidal-dpcp-jmlr-2018,Zhu-DPCP-nips-2018}. Let $Q \in \mathbb{R}^{m\times n}$ denote a matrix whose columns form an orthonormal basis of $W$. The FSV problem can be formulated as 
\begin{eqnarray}\label{sparse-vector-0}
\min \ \|Qx\|_0,\quad \mbox{s.t.} \ x\in S^{n-1},
\end{eqnarray}
where $S^{n-1} = \{x\in R^n\mid \|x\|=1\}$ is the unit sphere, and $\|z\|_0$ counts the number of nonzero entries of $z$. Because of the combinatorial nature of the cardinality function $\|\cdot\|_0$, \eqref{sparse-vector-0} is very difficult to solve in practice. In the literature, people have been focusing on its $\ell_1$ norm relaxation given below \cite{Qu-survey,Qu-FSV-TIT,Tsakiris-Vidal-dpcp-jmlr-2018,Zhu-DPCP-nips-2018}:  
\begin{eqnarray}\label{sparse-vector-1}
\min \ \|Qx\|_1,\quad \mbox{s.t.}\  x\in S^{n-1},
\end{eqnarray}
where $\|z\|_1:=\sum_i |z_i|$ is the $\ell_1$ norm of vector $z$.
Many algorithms have been proposed for solving \eqref{sparse-vector-1}, including the Riemannian gradient sampling algorithm \cite{Hosseini}, projected subgradient method \cite{Zhu-psgm-nips-2019}, Riemannian subgradient method \cite{Li}, manifold proximal point algorithm \cite{Chen-ManPPA-2020} and so on. 

Moreover, for the compressive sensing problems that have the same objective functions as \eqref{sparse-vector-0} and \eqref{sparse-vector-1}, people have found that using the $\ell_p$ quasi-norm $\|z\|_p$ ($0<p<1$) to replace $\|z\|_1$ can help promote the sparsity of $z$ \cite{Chartrand-Yin-Lp-cs,Foucart-Lai-2009,Chen-Ge-Wang-Ye-2014,XuF2010,Liu-Dai-Ma-2015,Liu-Ma-Dai-Zhang-2016}. Motivated by this, we propose to consider the following $\ell_p$ ($0<p<1$) minimization model for the FSV problem: 
\begin{eqnarray}\label{sparse-vector-p}
\min \ f(x) :=\|Qx\|_p^p,\quad \mbox{s.t.}\ x\in S^{n-1},
\end{eqnarray}
where $\|z\|_p^p := \sum_i |z_i|^p$. Note that algorithms proposed in  \cite{Hosseini,Zhu-psgm-nips-2019,Li,Chen-ManPPA-2020} for solving \eqref{sparse-vector-1} do not apply to \eqref{sparse-vector-p}, because the objective function in \eqref{sparse-vector-p} is non-Lipschitz. We propose to solve \eqref{sparse-vector-p} using our RSSD algorithm, and we now show the details. 

According to \cite{Absilbook}, the tangent space at  $x\in S^{n-1}$ is 
\begin{eqnarray*}
\T_x S^{n-1} :=\{z\in \mathbb{R}^n \mid x^\top z =0\},
\end{eqnarray*}
and the projection of $\xi\in \mathbb{R}^n$ onto the tangent space $\T_x S^{n-1}$ is 
\[
\operatorname{Proj}_{{\T_x}S^{n-1}} \xi = (I-x x^\top) \xi.
\]
In our RSSD algorithm, we use 
$
R_x(\xi) = {(x+\xi)}/{\|x+\xi\|}
$  
as the retraction function. 
We use the following smoothing function for \eqref{sparse-vector-p}:
\begin{eqnarray}\label{tilde-f}
\tilde f(x,\mu) = \sum_{i=1}^m [s_{\mu}((Qx)_i)]^p,
\end{eqnarray}
where $s_{\mu}(t)$ is the uniform smoothing function for $|t|$ defined in \eqref{uniform-sf}.

Note that our RSSD can also solve the $\ell_1$ norm minimization problem \eqref{sparse-vector-1}. Therefore, we compare our RSSD with two existing algorithms: Riemannian gradient sampling (RGS) method \cite{Hosseini1} and Riemannian nonsmooth BFGS (RBFGS) method \cite{Hosseini2} on \eqref{sparse-vector-1}. For the $\ell_p$ quisi-norm minimization problem \eqref{sparse-vector-p}, since no existing method is available for solving it, we only use our RSSD method to solve it, and we test RSSD with different $p \in(0,1)$ to see the effect of $p$ to the problem \eqref{sparse-vector-p}.  

The FSV problems are generated as follows. We choose $n\in\{5,10,15,20\}$ and $m \in \{4n,6n,8n,10n,12n,14n\}$. The subspace $W$ is generated following the same way as \cite{Hosseini1}. More specifically, we first generate the vector $e=(1,\ldots,1,0,\ldots,0)^\top$ whose first $n$ components are $1$ and the remaining $m-n$ components are $0$. We then generate $n-1$ random vectors in $\mathbb{R}^m$. The subspace $W$ is the span of $e$ and these $n-1$ random vectors. We use $Q\in\mathbb{R}^{m\times n}$ to denote the matrix whose columns form an orthonormal basis of $W$. The minimum value for $\|Q x\|_0$ on the sphere is likely to equal to $n$ in this case. 

We terminate our RSSD method if $\mu_l<10^{-6}$ and $\delta_l<10^{-4}$. As suggested in \cite{Hosseini2}, we terminate the RGS and the RBFGS methods if one of the following two conditions is satisfied:
\begin{itemize}
\item[(i)] the step size is less than the machine precision $2.22\cdot 10^{-16}$;
\item[(ii)] $\epsilon_k \le 10^{-6}$ and $\delta_k \le 10^{-12}$.
\end{itemize}
Moreover, the maximum number of iterations is set to $n_{\max} = 1000$ for all three methods. Our RSSD, as well as the RGS method, are implemented in MATLAB. The RBFGS codes were provided to us by Wen Huang, one of the authors of \cite{Hosseini2}, and they were written in C++ with a MATLAB interface. 
The parameters in the RGS and the RBFGS methods are set following the suggestions in \cite{Hosseini} and \cite{Hosseini2}. 
The parameters of our RSSD method are set as
\begin{eqnarray}\label{para-RSSD}
\mu_0 = 1,\ \delta_0 = 0.1, \theta_{\mu} = 0.5,\ \theta_{\delta} = 0.5.
\end{eqnarray}
We choose the initial points from normally distributed random vectors, using MATLAB code
\begin{eqnarray*}
{x^{0} = {\rm randn}(n,1); x^{0} = x^{0}/{\rm norm}(x^{0})}.
\end{eqnarray*}
For each $(m,n)$, we generate 50 random instances with 50 random initial points. We claim that an algorithm successfully finds the sparsest vector if  $\|Q \hat x\|_0 = n$ where $\hat x$ is the computed solution. Here, when we count the number of nonzeros of $Q \hat x$, we truncate the entries as
\[(Q\hat x)_i =0,\quad {\mbox{if}}\ |(Q\hat x)_i|< \tau,\]
where $\tau > 0$ is a pregiven tolerance.  
We report the number of successful cases out of 50 cases in Tables \ref{comparison-p1-n5} and \ref{comparison-p1-n15}. For RGS, we run the algorithm for 50 runs for each initial point and we also report the standard deviation of the average number of the successful cases.

Tables \ref{comparison-p1-n5} and \ref{comparison-p1-n15} record the number of success for the three methods for the $\ell_1$ model \eqref{sparse-vector-1} with different parameters $(m,n)$. The bold numbers in the tables highlight the largest number of success for the corresponding $(m,n)$. Comparing RGS, RBFGS and RSSD, we see that our RSSD method can provide a solution with the best accuracy, because when $\tau = 10^{-8}$, both the RGS and the RBFGS fail to recover the groundtruth,  but our RSSD method can still recover the groudtruth in many instances. 
From Tables \ref{comparison-p1-n5} and \ref{comparison-p1-n15} we see that, in the total 64 cases, RSSD performed the best in 39 cases. For other cases that RSSD is not the best, it is still comparable in most cases.

In Tables \ref{comparison-p1-n5} and \ref{comparison-p1-n15} we also report the results for RSSD-g, which incorporates a global technique to RSSD by selecting the best parameters $(\theta_{\mu}, \theta_{\delta})$ from a subset of choices. More specifically, it is worth mentioning that Example \ref{example-3.1} indicates that the different relations of the sequence of unknowns and the sequence of the smoothing parameters may yield different accumulation points. The parameter $\theta_{\mu}$ in our RSSD method controls the speed of the smoothing function that approaches to the original function, and the parameter $\theta_{\delta}$ determines the requirement of accuracy for the approximated solution along with the iterations. The different  relations of the two sequences can be obtained by using different choices of $(\theta_{\mu},\theta_{\delta})$. The number of successful instances can be improved if we tune the parameters $(\theta_{\mu},\theta_{\delta})$ for different settings of $(m,n)$. We record in the last column of Tables \ref{comparison-p1-n5} and \ref{comparison-p1-n15} the numbers of successful instances of our RSSD method by selecting the best result using the different choices of
\begin{eqnarray}\label{multiple}
(\theta_{\mu}, \theta_{\delta}) = (0.5,0.5), (0.1,0.5), (0.5,0.1), (0.8,0.2), (0.2,0.8).
\end{eqnarray} 
We see from Tables \ref{comparison-p1-n5} and \ref{comparison-p1-n15} that by selecting the best parameters in \eqref{multiple}, the performance of RSSD is clearly significantly improved.

\begin{table}[tbhp]
\label{comparison-p1-n5}
\footnotesize{
\begin{center}
\caption{Number of success  from 50 random initial points for the $\ell_1$ minimization model \eqref{sparse-vector-1} and $n=5, 10$
 (result for RGS  is  ``average number of success  $\pm$ standard deviation").}
\vskip 2mm
\renewcommand\arraystretch{1.1}
\addtolength{\tabcolsep}{3pt}
\begin{tabular}{c|c|c|c|c||c}
\hline
             $(n,m) $ & $\tau$ & {RGS} & {RBFGS} & {RSSD} & {RSSD-g}\\ \hline
              $(5,20)$ & $10^{-5}$ &  { ${\bf 19.96 \pm 1.442}$} & 16  & 16 & 22  \\
                         & $10^{-6}$ & { ${\bf19.96 \pm 1.442}$} & 0 &  16  & 22\\
                         & $10^{-7}$ & $0.36  \pm 0.598$ & 0 &  {\bf 16} & 22\\
                         & $10^{-8}$ & $0 \pm 0$ & 0 & {\bf 16} & 22 \\ \hline
              $(5,30)$ & $10^{-5}$ &  ${\bf 26.12 \pm 2.537}$ & 22  &  21 & 30 \\
                         & $10^{-6}$ & ${\bf 26.12 \pm 2.537}$ & 0  &  21 & 30 \\
                         & $10^{-7}$ & $0.62 \pm 0.667$ & 0 &  {\bf 21} & 30\\
                         & $10^{-8}$ & $0 \pm 0$ & 0 &  {\bf 2} & 30\\ \hline
              $(5,40)$ & $10^{-5}$ &  ${\bf 45.54 \pm 1.555}$ & 31  &  28 & 43  \\
                         & $10^{-6}$ & ${\bf 45.54 \pm 1.555}$ & 1 & 28  & 43\\
                         & $10^{-7}$ & $0.78 \pm 0.932$ & 0 & \bf 28 & 43 \\
                         & $10^{-8}$ & $0 \pm 0$ & 0 & \bf 28 & 43 \\ \hline
              $(5,50)$ & $10^{-5}$ &  ${\bf 50 \pm 0}$ & 46  &   49  & 50 \\
                         & $10^{-6}$ & ${\bf 50 \pm 0}$  & 44 &  49 & 50 \\
                         & $10^{-7}$ & $0.7 \pm 0.814$ & 26 & {\bf 49} & 50 \\
                         & $10^{-8}$ & $0 \pm 0$ & 0 &  {\bf 49} & 50 \\ \hline\hline 
              $(10,60)$ & $10^{-5}$ &  $24.16 \pm 2.17$ & {\bf 26}  &  25 & 38 \\
                         & $10^{-6}$ & $24 \pm 2.231$ & {\bf 26}  &  25  & 38 \\
                         & $10^{-7}$ & $0 \pm 0$ & 18 & {\bf 25} & 38 \\
                         & $10^{-8}$ & $0 \pm 0$ & 0 & {\bf 25} & 38 \\ \hline
              $(10,80)$ & $10^{-5}$ &  ${\bf 32.5 \pm 2.013}$ & 31  &  29 & 44 \\
                         & $10^{-6}$ & ${\bf 32.42 \pm 1.960}$ & 31 &  29  & 44 \\
                         & $10^{-7}$ & $0.002 \pm 0.141$ & 18 &  {\bf 29} & 44\\
                         & $10^{-8}$ & $0 \pm 0$ & 0 & {\bf 29} & 44\\ \hline
              $(10,100)$ & $10^{-5}$ &  ${\bf 44.68 \pm 2.035}$ & 40    & 44 & 48  \\
                         & $10^{-6}$ & $\bf 44.56 \pm 1.971$  & 40 &  44 & 48 \\
                         & $10^{-7}$ & $0.02 \pm 0.141$ & 24 & \bf 44 & 48 \\
                         & $10^{-8}$ & $0 \pm 0$ & 0 &  \bf 44 & 48 \\ \hline
              $(10,120)$ & $10^{-5}$ &  $ \bf 46.44\pm 1.756$ & 41  &  36  & 46\\
                         & $10^{-6}$ & $ \bf 46.22\pm 1.753$  & 41  &  36 & 46 \\
                         & $10^{-7}$ & $  0.1 \pm  0.303$ & 31 & \bf 36 & 46\\
                         & $10^{-8}$ & $ 0 \pm 0$ & 0 &  {\bf 36} & 46 \\ \hline
\end{tabular}
\end{center}}
\end{table}

\begin{table}[tbhp]
\label{comparison-p1-n15}
\footnotesize{
\begin{center}
\caption{Number of success  from 50 random initial points for the $\ell_1$ minimization model \eqref{sparse-vector-1} and $n=15, 20$
 (result for RGS  is ``average number of success  $\pm$ standard deviation").}
\vskip 2mm
\renewcommand\arraystretch{1.1}
\addtolength{\tabcolsep}{3pt}
\begin{tabular}{c|c|c|c|c||c}
\hline
             $(n,m) $ & $\tau$ & {RGS} & {RBFGS} & {RSSD} & {RSSD-g}\\ \hline 
              $(15,90)$ & $10^{-5}$ &  $12.16 \pm 2.054$ & 9  & \bf 16 & 31   \\
                         & $10^{-6}$ & $12.16 \pm 2.054$ & 9  & \bf  16 & 31 \\
                         & $10^{-7}$ & $0 \pm 0$ & 4 & \bf 16 & 31 \\
                         & $10^{-8}$ & $0 \pm 0$ & 0 & \bf 16 & 31 \\ \hline
              $(15,120)$ & $10^{-5}$ &  $16.78 \pm 2.401$ & 17  &  \bf 20 & 36 \\
                         & $10^{-6}$ & $16.74 \pm 2.319$ & 17 &  \bf 20 & 36 \\
                         & $10^{-7}$ & $0 \pm 0$ & 11 & \bf 20 & 36  \\
                         & $10^{-8}$ & $0 \pm 0$ & 0 & \bf 20 & 36 \\ \hline
              $(15,150)$ & $10^{-5}$ &  $36.84 \pm 1.856$ & 40  &  \bf 41  & 48\\
                         & $10^{-6}$ & $36.8 \pm 1.906$  & 40 & \bf 41 & 48 \\
                         & $10^{-7}$ & $0 \pm 0$ & 36 & \bf 41 & 48 \\
                         & $10^{-8}$ & $0 \pm 0$ & 0 & \bf  41 & 48 \\ \hline
              $(15,180)$  & $10^{-5}$ &   $26.76\pm 2.162$ & {\bf 33} &  26 & 40 \\
                         & $10^{-6}$ & $ 26.66\pm 2.125 $  & {\bf 33} &  26 & 40 \\
                         & $10^{-7}$ & $ 0\pm 0$ & {\bf 31} & 26 & 40 \\
                         & $10^{-8}$ & $ 0\pm 0$ & 0  & \bf 26 & 40 \\ \hline\hline
               $(20,160)$ & $10^{-5}$ &  $19.64 \pm 2.819$ & 28  &  \bf 41  & 43\\
                         & $10^{-6}$ & $19.62 \pm 2.849$ & 28 & \bf  41 & 43\\
                         & $10^{-7}$ & $0  \pm 0$ & 20 &  \bf 41 & 43 \\
                         & $10^{-8}$ & $0 \pm 0$ & 0 &  \bf 38 & 43\\ \hline
              $(20,200)$ & $10^{-5}$ &  $20.74 \pm 2.717$ & 25  &  \bf 29  & 46\\
                         & $10^{-6}$ & $20.74 \pm 2.717$ & 24  & \bf 29 & 46 \\
                         & $10^{-7}$ & $0 \pm 0$ & 23 & \bf 29 & 46\\
                         & $10^{-8}$ & $0 \pm 0$ & 0 & \bf 29 & 46\\ \hline
              $(20,240)$ & $10^{-5}$ &  $20.60 \pm 2.441$ & {\bf 30}  &  24  & 35\\
                         & $10^{-6}$ & $20.58 \pm 2.400$ & {\bf 30} & 24  &35\\
                         & $10^{-7}$ & $0 \pm 0$ & {\bf 30} & 24 & 35 \\
                         & $10^{-8}$ & $0 \pm 0$ & 0 & {\bf 24} & 35\\ \hline
              $(20,280)$ & $10^{-5}$ &  $ 24.62\pm 2.230$ & {\bf 32}  &  27 & 37  \\
                         & $10^{-6}$ & $24.60 \pm 2.222$  & {\bf 32} & 27 & 37 \\
                         & $10^{-7}$ & $0 \pm 0$ & {\bf 32} & 27 & 37\\
                         & $10^{-8}$ & $0 \pm 0$ & 0 &  {\bf 27} & 37 \\ \hline
\end{tabular}
\end{center}}
\end{table}

We now report the results of solving the $\ell_p$ minimization model \eqref{sparse-vector-p} using our RSSD method. In Tables \ref{comparison1} and \ref{comparison} we again report the number of successes from 50 random instances. Here we only report the results for $\tau = 10^{-8}$. We also include the results for the $\ell_1$ minimization model \eqref{sparse-vector-1} for the purpose of comparison. Note that Table \ref{comparison} correponds to the RSSD-g, i.e., RSSD with parameters $(\theta_{\mu},\theta_{\delta})$ chosen as the best one in \eqref{multiple}.  
From Tables \ref{comparison1} and \ref{comparison} we see that the $\ell_p$ minimization model \eqref{sparse-vector-p} can indeed be better than the $\ell_1$ minimization model \eqref{sparse-vector-1}, as long as an appropriate $p$ is used.

\begin{table}[tbhp]
\label{comparison1}
\footnotesize{
\begin{center}
\caption{Number of success among runs from 50 random initial points
 for the $\ell_1$ minimization model \eqref{sparse-vector-1} and the $\ell_p$ minimization model \eqref{sparse-vector-p} with $p = 0.9,\ldots,0.1$, with $\tau = 10^{-8}$ by our RSSD method with $(\theta_{\mu},\theta_{\delta}) = (0.5,0.5)$.}
\vskip 2mm
\renewcommand\arraystretch{1.1}
\addtolength{\tabcolsep}{3pt}
\begin{tabular}{c|c|ccccccccc}
\hline
\multirow{2}{*}{$(m,n)$} & \multirow{2}{*}{$\ell_1$} & \multicolumn{9}{c}{$\ell_p$ minimization model with  $0<p< 1$}\\ \cline{3-11}
          &    & $0.9$ & $0.8$ & $0.7$ & $0.6$ & $0.5$ & $0.4$ & $0.3$ & $0.2$ & $0.1$\\ \hline
 $ (5,20)$ 
               & 16 & 17  &  17 & 19 &   17 & 19 & 19 & 19 & \bf 20 & \bf 20 \\
               
$(5,30)$ & 0 & 21 & \bf 22  & \bf 22 & 21 & 0 & \bf 22 & \bf 22 & \bf 22 & \bf 22 \\
                
$(5,40)$  & 28 & 29  &  \bf 35 & 33 &   28 & 31 & 31 & 30 & 30 & 33\\
                
$(5,50)$ & 49 & 49  &  \bf 50 & 49 &  49 & 49 & \bf 50 & 49 & 49 & 48 \\ \hline
               
$ (10,60)$
 
              & 25 & 27  &  \bf 28 & 26 &   25 & 25 & 25 & 25 & 25 & 25  \\
               
$(10,80)$
 
             & 29 & 27  & \bf 31 & 30 &  28 & 28 & 28 & 29 & 30 & 28 \\
                
$(10,100)$
 
             & \bf 44 & 43  &  42 & \bf 44 &  41 & 43 & 43 & 43 & \bf 44 & 43\\
                 
$(10,120)$
 
             & 36 & 35  &  35 & 37 &  38 & 35 & 37 & 38 & 36 & \bf 41 \\ \hline
 $(15,90)$
  
                & 16 & 16  &  18 & 18 &   18 & \bf 19 & \bf 19 & 17 & 16 & 16   \\
               
$(15,120)$
 
               & 20 & 21  & 23 & 19 &  21 & 24 & 23 & 24 & \bf 25 & 21 \\
                
$ (15,150)$
 
&  41 & 43  &  \bf 44 & 39 &  35 & 38 & 38 & 37 & 35 & 38\\
                 
$ (15,180)$
 
  & 26 & 26  &  26 & \bf 30 &  29 & 26 & 27 & 26 & 26 & 26  \\ \hline
 $(20,160)$
  
                & \bf 41 & 40  &  \bf 41 & 36 &  33 & 34 & 37 & 38 & 39 & 39   \\
               
$(20,200)$
 
               & 29 & 29  & 30 & \bf 33 &  \bf 33 & \bf 33 & 30 & 30 & 30 & 27 \\
                
$ (20,240)$
 
                & \bf 24 & 22  & 23 & 20 &  20 & 21 & 21 & 19 & 19 & 21\\
                 
$ (20,280)$
 
                & 27 & 26  &  28 & 27 &  \bf 29 & 25 & 24 & 24 & 24 & 26 \\
 \hline
\end{tabular}
\end{center}}
\end{table}

\begin{table}[tbhp]
\label{comparison}
\footnotesize{
\begin{center}
\caption{Number of success among runs from 50 random initial points
for the $\ell_1$ minimization model \eqref{sparse-vector-1} and the $\ell_p$ minimization model \eqref{sparse-vector-p} with $p = 0.9,\ldots,0.1$, with $\tau = 10^{-8}$ by our RSSD-g method, i.e., RSSD method using multiple choices $(\theta_{\mu},\theta_{\delta})$ in \eqref{multiple}.}
\vskip 2mm
\renewcommand\arraystretch{1.1}
\addtolength{\tabcolsep}{3pt}
\begin{tabular}{c|c|ccccccccc}
\hline
\multirow{2}{*}{$(m,n)$} & \multirow{2}{*}{$\ell_1$} & \multicolumn{9}{c}{$\ell_p$ minimization model with  $0<p< 1$}\\ \cline{3-11}
          &    & $0.9$ & $0.8$ & $0.7$ & $0.6$ & $0.5$ & $0.4$ & $0.3$ & $0.2$ & $0.1$\\ \hline
 $ (5,20)$
  
               & 22 & 21 & 20 & 22 & 22 & 23 & 23 & \bf 25 & 23 & 21 \\
               
$(5,30)$ 
         & 30 & 30 & 31 & 32 & 32 & 30 & 32 & 31 & \bf 35 & 27 \\
                
$(5,40)$  
          & 43 & 40 & 41 & 43 & 43 & \bf 44 & 42 & 43 & 42 & 42\\
                 
$(5,50)$
         & \bf 50 & \bf 50 & \bf 50 & \bf 50 & \bf 50 & \bf 50 & \bf 50 & \bf 50 & \bf 50 & \bf 50\\ \hline
               
$ (10,60)$
 
              & 38 & 39 & 41 & 37 & 39 & 39 & 38 & \bf 42 & 39 & 38\\
               
$(10,80)$
 
             & 44 & 42 & 43 & 42 & 42 & \bf 45 & 43 & \bf 45 & 43 & 40\\
                
$(10,100)$
 
             & 48 & 47 & 48 & 48 & \bf 49 & 48 & 48 & 48 & 48 & 47\\
                 
$(10,120)$
 
             & 46 & 47 & 46 & 46 &  46 & 46 & 44 & 45 & 46 & \bf 48 \\ \hline
 $(15,90)$
  
                & 31 & 26 & 26 & 31 & 30 & 28 & 32 & \bf 33 & 30 & 31\\
               
$(15,120)$
 
               & 35 & \bf 39 & 36 & 35 & 36 & 35 & 35 & 33 & 34 & 32\\
                
$ (15,150)$
 
&  \bf 48 & 47 & \bf 48 & 47 & \bf 48 & 47 & 47 & \bf 48 & \bf 48 & 45\\
                 
$ (15,180)$
 
  & 40 & 41 & 40 & 42 & 42 & 42 & 40 & 40 & \bf 43 & 38 \\ \hline
 $(20,160)$
  
                & 43 & \bf 45 & 44 & 41 & 42 & 44 & 43 & 43 & \bf 45 & 43\\
              
$(20,200)$
 
               & \bf 46 & \bf 46 & 45 & 45 & 43 & 45 & 45 & 40 & 41 & 41\\
                
$ (20,240)$
 
                & 35 & 32 & 31 & 33 & 34 & 35 & 34 & 32 & 35 & \bf 36 \\
                 
$ (20,280)$
 
                & 35 & 38 & 38 & 38 & 39 & \bf 41 & 40 & \bf 41 & 40 & 38 \\
 \hline
\end{tabular}
\end{center}}
\end{table}

\subsection{Sparsely-used orthogonal complete dictionary learning}

Given a set of data $Y = [{\bf y}_1,{\bf y}_2,\ldots,{\bf y}_m] \in \mathbb{R}^{n\times m}$, the sparsely-used
orthogonal complete dictionary learning (ODL) seeks a dictionary that can sparsely represent $Y$.  
More specifically, ODL seeks an orthogonal matrix $X = [{\bf x}_1,{\bf x}_2,\ldots,{\bf x}_n] \in \mathbb{R}^{n\times n}$ and a sparse matrix $S\in \mathbb{R}^{n\times m}$ such that $Y \approx X S$. The matrix $X$ is called an orthogonal dictionary. We refer to \cite{JuSun} for more details of this model. This problem can be modeled as an $\ell_0$ minimization problem \cite{Spielman}: 
\begin{eqnarray}\label{od0}
 \min \ \frac{1}{m} \sum_{i=1}^m  \|{\bf y}_i^\top X\|_0,\quad {\rm s.t.}\ X \in {\rm St}(n,n),
\end{eqnarray}
where ${\rm St}(n,n) = \{X \in \mathbb{R}^{n\times n}\ :\ X^\top X = I_n\}$ is the Stiefel manifold. To overcome the computational difficulty of the $\ell_0$ minimization model, the $\ell_0$ term is usually replaced by the $\ell_1 $ norm in the literture, which leads to the following $\ell_1$ minimization problem for ODL \cite{Spielman,JuSun}:
\begin{eqnarray}\label{od1}
\min \ \frac{1}{m} \sum_{i=1}^m
 \|{\bf y}_i^\top  X\|_1,\quad {\rm s.t.}\ X \in {\rm St}(n,n).
\end{eqnarray}
Here we again consider the $\ell_p$  ($0<p<1$) quasi-norm minimization model
\begin{eqnarray}\label{odp}
\min \ \frac{1}{m}\sum_{i=1}^m \|{\bf y}_i^\top  X\|_p^p, \quad {\rm s.t.}\ X \in {\rm St}(n,n),
\end{eqnarray}
and apply our RSSD method to solve it.  
We now specify the details. The tangent space of the Stiefel manifold is 
\[\T_{X}{{\rm St}(n,n)} := \{\xi \in \mathbb{R}^{n\times n} \ :\ \xi^\top X + X^\top \xi = 0\}.\] 
We use the QR factorization as the retraction on the Stiefel manifold, which is given by
$R_{X}(\xi) = {\rm qf}(X+\xi)$. Here ${\rm qf}(A)$ denotes the $Q$ factor of the QR decomposition of $A$.

In \cite{Li}, Li \etal proposed a Riemannian subgradient method and its variants -- Riemannian incremental subgradient method and Riemannian stochastic subgradient method -- for solving the $\ell_1$ minimization problem \eqref{od1}. In this section, we compare our RSSD for solving the $\ell_p$ minimization problem \eqref{odp} and compare its performance with the algorithms proposed in \cite{Li} for solving \eqref{od1}. We thus generate the synthetic data for ODL in a similar manner as \cite{Li}, which is detailed below. We first generate the underlying orthogonal dictionary ${X}^*\in {\rm St}(n,n)$ with $n =30$ whose entries are drawn according to standard Gaussian distribution. The number of samples $m = \lfloor 10 \cdot n^{1.5}\rfloor=1643$. The sparse matrix ${S}^* \in\mathbb{R}^{n\times m}$ is generated such that the entries follow the Bernoulli-Gaussian distribution with parameter 0.5. Finally, we set $Y = X^*S^*$. We generate 50 instances using this procedure. For each instance, we generate two different initial points: one is a standard Gaussian random vector denoted as $x_{0}^{\rm Gauss}$, and the other one is a uniform random vector denoted as $x_0^{\rm uniform}$.  
For the ease of presentation, we denote the three algorithms in \cite{Li} -- Riemannian subgradient method, Riemannian incremental subgradient method, and Riemannian stochastic subgradient method -- as R-Full, R-Inc and R-Sto, respectively.  
We use our RSSD to solve the $\ell_p$ minimization model \eqref{odp} with $p=0.001$. Moreover, we again truncate the entries of $Y^\top \hat X$ as
\begin{eqnarray*}
 (Y^\top \hat X)_{ij} =0, \quad \mbox{if}\ |(Y^\top \hat X)_{ij}|< \tau,
\end{eqnarray*}
where $\tau>0$ is a pregiven tolerance, and $\hat{X}$ is the computed solution.  
We use the same parameters in \eqref{para-RSSD} for RSSD. The codes for R-Full, R-Inc and R-Sto were downloaded from the author's webpage\footnote{\url{https://github.com/lixiao0982/Riemannian-subgradient-methods}.}.

{All the algorithms were run in MATLAB (R2018b) on a notebook with 1.80GHz CPU and 16GB of RAM.  For each instance, we terminated the algorithm when the CPU time reaches 50 seconds. 
} 
We report the average of the sparsity level of $Y^\top \hat X$ over 50 instances in Table \ref{ODL}, where the sparsity level is computed by 
\begin{eqnarray*}
{\textbf{sparsity level}} = \frac{{\text{number of zero entries of}}\ Y^\top \hat X }{mn}.
\end{eqnarray*}
Note that the desired sparsity level of $Y^\top \hat X$ is 0.5 because of the way that  $S^*$ was generated. We see from Table \ref{ODL} that the $\ell_p$ minimization model with $p=0.001$ solved by our RSSD method provides the best results in terms of the sparsity level. 

{Moreover, we plot the trajectory of the sparisty level in Figures \ref{fig:Gaussian-accuvscpu} and \ref{fig:Uniform-accuvscpu}. From these figures, it is clear that the $\ell_p$ minimization model (\ref{odp}) with $p=0.001$ solved by our RSSD method provides the best results in terms of sparsity level. More specifically, our RSSD method can improve the sparsity to the desired level, while the other three algorithms stopped making progress after about one second.}

\begin{table}[tbhp]
\label{ODL}
\footnotesize{
\begin{center}
\caption{Average of sparsity levels of computed solutions from 50 instances 
}
\renewcommand\arraystretch{1.1}
\addtolength{\tabcolsep}{3pt}
\begin{tabular}{c|ccc|c}
\hline
\multirow{2}{*}{Initial points} & \multicolumn{3}{c}{$\ell_1$ minimization model} \vline &  $\ell_p$ model,  $p=0.001$\\ \cline{2-5}
                                & R-Full & R-Inc & R-Sto &  RSSD\\ \hline
$ x_0^{\rm Gauss}$, $\tau = 10^{-4}$ & 0.3727 & 0.3857  &  0.3456 &\bf 0.5000  \\
$ x_0^{\rm Gauss}$, $\tau = 10^{-5}$ & 0.3697 & 0.3852  &  0.3450 &\bf 0.4895  \\ \hline
$x_0^{\rm Uniform}$, $\tau = 10^{-4}$ & 0.3727 & 0.3784  & 0.3234 &\bf 0.5000 \\
$x_0^{\rm Uniform}$, $\tau = 10^{-5}$ & 0.3675 & 0.3773  &  0.3222 &\bf 0.4915\\ \hline
\end{tabular}
\end{center}}
\end{table}

\begin{figure}[tbhp] 
\centering 
 {\includegraphics[width=0.4\textwidth]{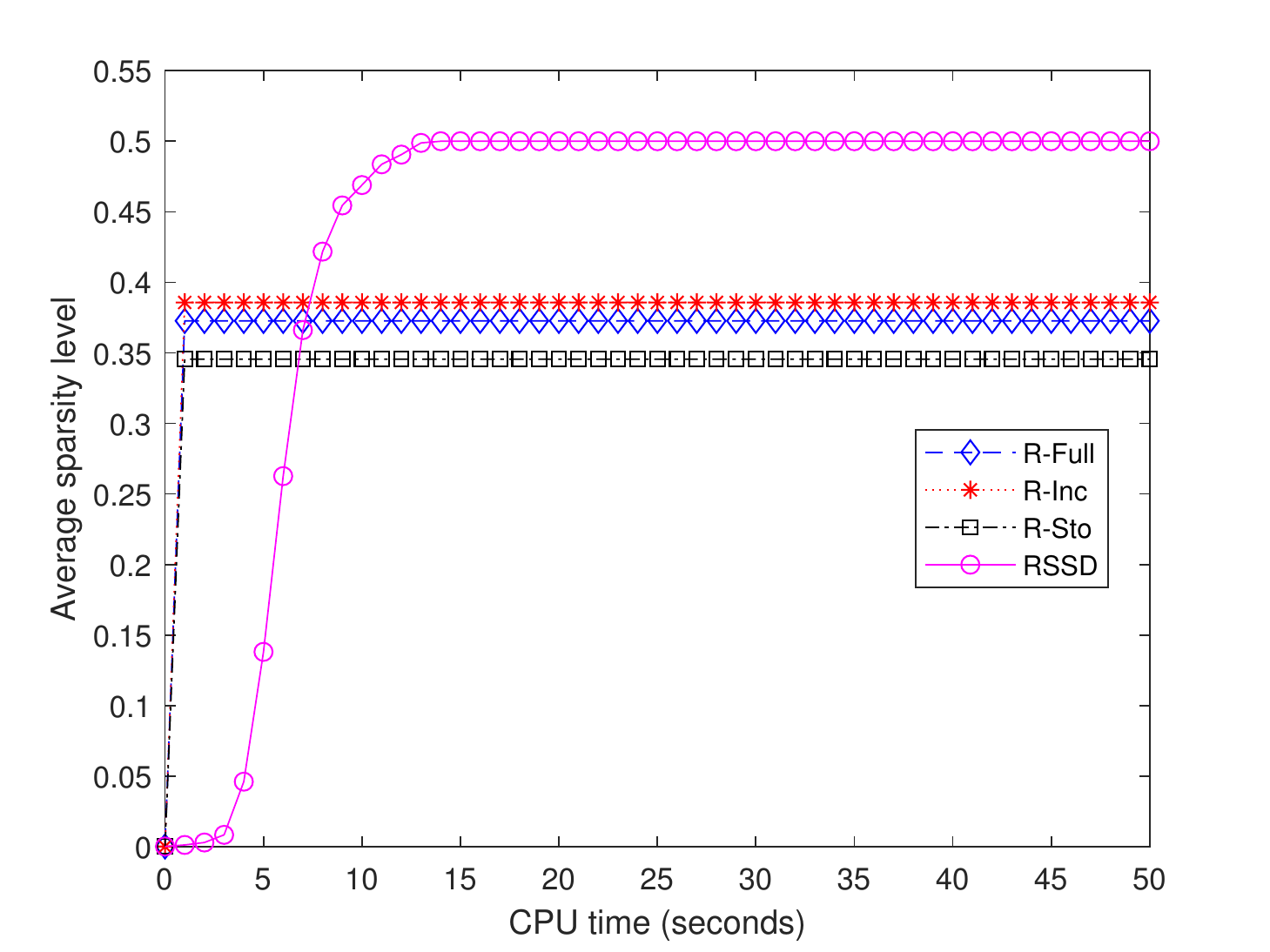}
 \includegraphics[width=0.4\textwidth]{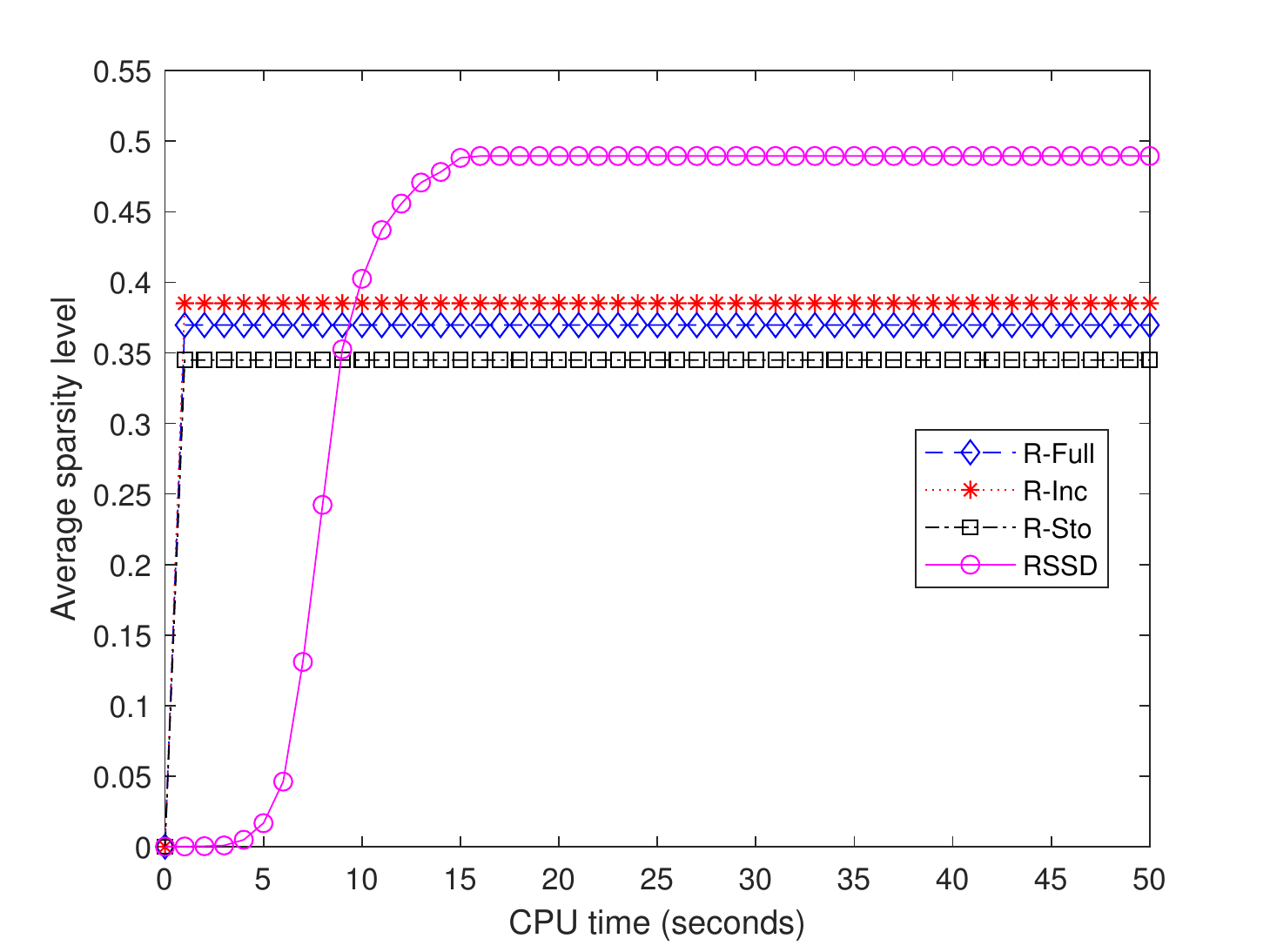}}
  \par\vspace{0pt}
{\caption{
\small{Average sparsity level versus CPU time of 50 instances using Guassian initial points. Left: $\tau = 10^{-4}$; Right: $\tau = 10^{-5}$}. }\label{fig:Gaussian-accuvscpu}}
\end{figure}

\begin{figure}[tbhp] 
\centering 
 {\includegraphics[width=0.4\textwidth]{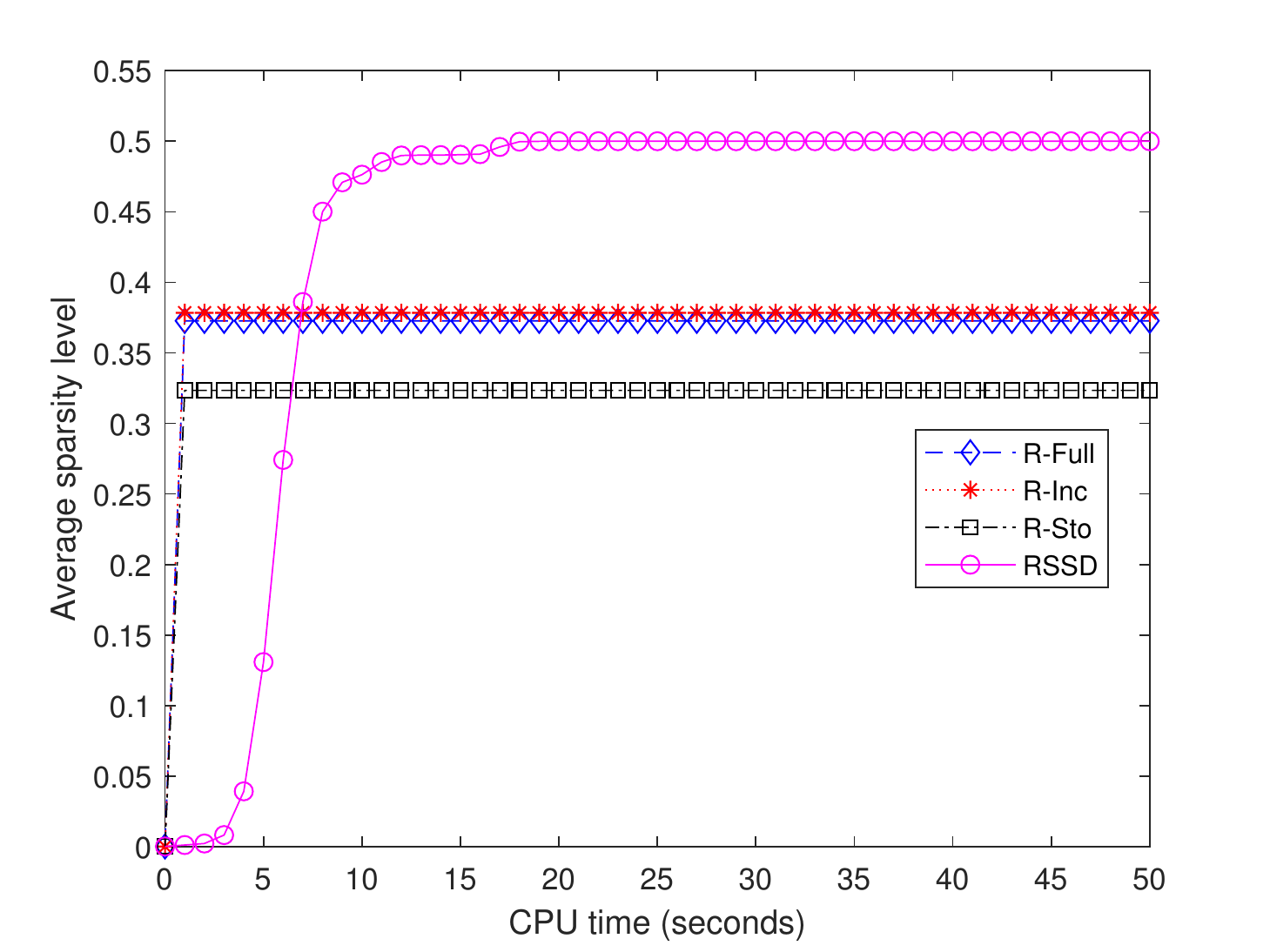}
 \includegraphics[width=0.4\textwidth]{Fig2-Uniform-tau1-eps-converted-to}}
  \par\vspace{0pt}
{\caption{
\small{Average sparsity level versus CPU time of  50 instances using uniform initial points. Left: $\tau = 10^{-4}$; Right: $\tau = 10^{-5}$}. }\label{fig:Uniform-accuvscpu}}
\end{figure}

\section{Concluding remarks}\label{sec:conclude}

In this paper, we developed RSSD, a novel Riemannian smoothing steepest descent method, for minimizing a non-Lipschitz function over Riemannian submanifolds. We studied some useful concepts such as the Riemannian generalized subdifferentials, and Riemannian gradient sub-consistency. We proved that any accumulation point generated by our RSSD method is a stationary point associated with the smoothing function employed in the method, which is necessary for local optimality of \eqref{model}. Moreover, under the Riemannian gradient sub-consistency, we also proved that any accumulation point is a limiting stationary point of \eqref{model}. Numerical results on finding a sparse vector in a subspace and the sparsely-used orthogonal complete dictionary learning demonstrate the advantage of the non-Lipschitz minimization models and the efficiency of our RSSD method.

\section*{Acknowledgements} We are very grateful to Professor Wen Huang of Xiamen University for providing the C++ code for the Riemannian BFGS method, and Hui Shi for helps on the numerical experiments.

\end{document}